\documentclass[a4paper]{amsart}

\usepackage{microtype,verbatim}

\usepackage{amsmath}
\usepackage{amsthm}
\usepackage{amssymb}
\usepackage[foot]{amsaddr}
\usepackage{multirow}
\usepackage{mathtools}

\AtBeginDocument{
  \renewcommand{\setminus}{\mathbin{\backslash}}%
}

\usepackage{tikz}
\usetikzlibrary{arrows}
\usetikzlibrary{cd}
\usepackage{array}
\usepackage{physics}
\usepackage{xcolor}
\usepackage{graphicx}
\usepackage{subcaption}
\usepackage{booktabs}
\usepackage{biblatex}
\usepackage{hyperref}
\usepackage[capitalise,noabbrev]{cleveref} 

\newtheorem{thm}{Theorem}
\numberwithin{thm}{section}
\Crefname{thm}{Theorem}{Theorems}
\newtheorem{lem}[thm]{Lemma}
\Crefname{lem}{Lemma}{Lemmata}
\newtheorem{prp}[thm]{Proposition}
\Crefname{prp}{Proposition}{Propositions}
\newtheorem{cor}[thm]{Corollary}

\newtheorem*{claim*}{Claim}
\theoremstyle{definition}
\newtheorem{defn}[thm]{Definition}
\newtheorem{ex}[thm]{Example}
\newtheorem{problem}[thm]{Problem}

\theoremstyle{remark}
\newtheorem*{rem}{Remark}

\newtheorem*{notation}{Notation}

\newcommand{\union}{\cup}
\newcommand{\inter}{\cap}

\newcommand{\mc}{\mathcal}

\renewcommand{\H}{\mathbb{H}}

\newcommand{\R}{\mathbb{R}}
\newcommand{\C}{\mathbb{C}}

\newcommand{\Z}{\mathbb{Z}}
\newcommand{\Q}{\mathbb{Q}}

\newcommand{\fib}{\mathsf{fib}}

\DeclareMathOperator{\PSL}{PSL}
\DeclareMathOperator{\SL}{SL}

\DeclareMathOperator{\Mod}{Mod}
\DeclareMathOperator{\ceil}{ceil}
\DeclareMathOperator{\floor}{floor}

\makeatletter
\@namedef{subjclassname@2020}{%
  \textup{2020} Mathematics Subject Classification}
\makeatother

\title{The combinatorics of Farey words and their traces}
\author{Alex Elzenaar}
\author{Gaven Martin}
\author{Jeroen Schillewaert}
\address[AE]{Max Planck Institute for Mathematics in the Sciences, Leipzig, Germany}
\address[GM]{Institute for Advanced Study, Massey University, Auckland, New Zealand}
\address[JS]{Department of Mathematics, The University of Auckland, New Zealand}
\email[AE]{aelz176@aucklanduni.ac.nz}
\email[GM]{g.j.martin@massey.ac.nz}
\email[JS]{j.schillewaert@auckland.ac.nz}


\begin{document}
\begin{abstract}
We introduce a family of $3$-variable ``Farey polynomials'' that are closely connected with the geometry and topology of $3$-manifolds and orbifolds as they can be used to produce concrete realisations of the boundaries and local coordinates for one-dimensional (over $ \C $) deformation spaces of Kleinian groups.  As such, this family of polynomials has a number of quite remarkable properties. We study these polynomials from an abstract combinatorial viewpoint, including a recursive definition extending that which is known in the literature for the special case of manifolds, even beyond what the geometry predicts. We also present some intriguing examples and conjectures which we would like to bring to the attention of researchers interested in algebraic combinatorics and hypergeometric functions.

The results in this paper additionally provide a practical approach to various classification problems for rank-two subgroups of $\PSL(2,\C)$  since they, together with other recent work of the authors, make it possible to provide certificates that certain groups are discrete and free,  and effective ways to identify relators.
\end{abstract}

\maketitle

\section{Introduction}
We introduce a fascinating family of $3$-variable polynomials $\Phi^{a,b}_{p/q}(z)$ defined for all $ a, b \in \C \union \{\infty\} $ (though the geometric meaning of these polynomials only makes sense for $ a $ and $ b $ that are integers that are at least $2$ or $ \infty $ and which are not both $ 2 $) and all rational numbers $p/q \in \Q \union \{\infty\}$. The goal of this paper is to study the abstract combinatorial properties of these polynomials, which we call \textbf{Farey polynomials}. These polynomials have been studied before in the case $ a = b = \infty $, since they appear naturally in studies of the Riley slice of Schottky space, a much studied model for one complex dimensional deformation spaces of Kleinian groups; even in that case many of our results or their proofs seem to be new.

We will now explain the mathematical background behind this paper in more detail.
A \textbf{Kleinian group} is a discrete subgroup of $ \PSL(2,\C) $; these groups have been intensively studied for a long time in association with hyperbolic geometry and
conformal geometry, because they act naturally on hyperbolic space $ \H^3 $ as isometries, and on its visual boundary $ \partial \H^3 = \C \union \{\infty\} $ as conformal
maps. Given a Kleinian group $ G $, let $ \Omega(G) $ be its domain of discontinuity in $ \C \union \{\infty\} $; then $ \H^3/G $ is a hyperbolic 3-fold and, if $ \Omega(G) \neq \emptyset $,
then $ \Omega(G)/G $ is an orbifold surface which can be identified with the boundary of this 3-fold. We assume that the reader is familiar with the basic theory of M\"obius transformations;
all standard terms which we use but do not define may be found in the textbook \autocite{beardon}.

Up to conjugacy in $ \PSL(2,\C) $, every group on two parabolic generators may be written as
\begin{displaymath}
  \Gamma = \left\langle X = \begin{bmatrix} 1&1 \\ 0&1 \end{bmatrix}, Y = \begin{bmatrix} 1 & 0 \\ z & 1 \end{bmatrix} \right\rangle;
\end{displaymath}
Much interest has surrounded the behaviour of these groups for different $ z $. When $ \abs{z} > 4 $, this group is discrete and free;
however, there are many $ z $ inside the radius $4$ circle which also give free and discrete groups. The set of all $ z $ for which $ \Gamma $
is free, discrete, and non-elementary is a connected subset of $ \hat{\C} $; in fact, it is isotopic to the punctured disc $ \abs{z} \geq 4 $. The interior of
this set is called the \textbf{Riley slice}, and we discuss the geometric interest of this set in \cref{sec:context}. However, for the remainder
of this paper, the reader need only know that ``$ z $ is in the Riley slice'' means that the group $ \Gamma $ is discrete, free, and has
quotients $ \H^3/\Gamma $ and $ \Omega(\Gamma)/\Gamma $ homeomorphic to respectively a 3-ball with two tubes drilled out and a sphere with four punctures.

Punctured spheres can be viewed is a limiting case of cone-pointed spheres as the cone angles tend to zero. It is therefore natural to also study the case of hyperbolic orbifolds that
are 3-balls with two finite-order singular arcs (so with conformal boundaries that are spheres with four cone points joined in pairs). This corresponds to replacing the parabolic generators
of $ \Gamma $ with elliptics, producing the three-parameter family of groups
\begin{equation}\label{eq:rileygp}
  \Gamma_z^{a,b} = \left\langle X = \begin{bmatrix} \alpha&1 \\ 0&\alpha^{-1} \end{bmatrix}, Y_z = \begin{bmatrix} \beta & 0 \\ z & \beta^{-1} \end{bmatrix} \right\rangle;
\end{equation}
where $ \alpha = \exp(\pi i/a) $ and $ \beta = \exp(\pi i/b) $ (so $ X $ and $ Y_z $ have respective orders $ a $ and $ b $) and $ z $ is some element of $ \C $
such that  $ \Gamma_z^{a,b} $ is discrete and isomorphic to $ \Z/a\Z * \Z/b\Z $.

By the standard theory of coverings by discrete groups, elements of $ \Gamma = \Gamma_z^{a,b} $ correspond to free homotopy classes of simple closed curves in the quotient
$ \Omega(\Gamma)/\Gamma$. In this paper we study certain words in the group $ \Gamma $ that are defined combinatorially in \cref{sec:cutting}, called Farey words. We are
motivated to study these words because, for $ z $ in the Riley slice, their conjugacy classes are exactly the conjugacy classes in the group that represent free homotopy
classes of simple non-boundary-parallel curves on the four-times punctured sphere $ \Omega(\Gamma)/\Gamma $ that do not bound compression discs in the 3-manifold when the surface
is viewed as the conformal boundary of $ \H^3/\Gamma $.

\begin{notation}
  In the remainder, we will often talk of curves when we mean free homotopy classes of simple closed curves. All of the formal results we will prove are results on the combinatorics of these free homotopy classes
  themselves (so it is all right to pick any specific curve in the free homotopy class, so long as it is in general position) or on the free groups uniformising the hyperbolic 3-fold (in which case the words represent free homotopy
  classes directly). Note that the surface $ \Omega(\Gamma)/\Gamma $ carries only a conformal metric, and so it is not normally possible to talk about geodesics. However, there is a sense in which there is
  a natural measure of length on these free homotopy classes; this length plays an important role in the broader theory, but here we need not worry about it except when we want to give some intuition: we will occasionally
  say that the length of a curve `goes to zero', and by this we mean only the purely topological fact that in the limit the surface curve becomes boundary parallel (surrounding a puncture).
\end{notation}

The abstract setting in which we work is a formalised version of this setup. Throughout:
\begin{itemize}
  \item $ z \in \C\setminus\{0\} $ is a complex parameter (which we do not assume lies in a Riley slice unless otherwise stated);
  \item $a $ and $ b $ are real numbers (not necessarily integral) or infinity: we write $ a,b \in \hat{\R} \coloneq \R \union \{\infty\} $
        and use `hats' in this way for all subsets of $ \C $, so for instance $ \hat{\C} $ is the Riemann sphere;
  \item $ \alpha = \exp(\pi i/a) $ and $ \beta = \exp(\pi i/b) $;
  \item and $ \Gamma_z^{a,b}  $ is the group generated by the elements $ X $ and $ Y_z $ defined symbolically in \cref{eq:rileygp}.
\end{itemize}
To simplify notation when $ z $ is fixed we often drop the subscripts, so $ Y = Y_z $ and $ \Gamma = \Gamma_z^{a,b} $. We also
use the convention $ x \coloneq X^{-1} $, $ y \coloneq Y^{-1} $ (and we use similar conventions throughout without comment when writing group words).

\subsection{The recursion formula}
Our main result is Theorem \ref{thm:recursion}, which gives a straightforward recursion formula for these Farey polynomials that is independent of the geometry and topology from which they naturally arise. These formulas make it possible to effectively compute deep into rank-two Kleinian groups to look at special words (Farey words) conjecturally describing the relators in a non-free group. Similar recursions for
the commutators $ [X^n,Y]-2 $ and the traces $ \tr^2(X^n)-4 $ were found by Alaqad, Gong, and Martin in \autocite{martin21}, and for the analogues of the
Farey polynomials for the Maskit slice in \autocite[283-285]{indras_pearls}. Links between these Farey words and other combinatorial words which occur in Teichm\"uller theory have been surveyed by Gilman \autocite{gilman06}.

\begin{ex}\label{ex:special_case}
  We describe the interesting special case of our formula where $ \alpha = \beta = \exp(i\theta) $ (so $ a = b = \pi/\theta $: note that they are not necessarily
  integral). Write $ f_{p/q} $ for $ \Phi^{\pi/\theta,\pi/\theta}_{p/q} $. Then our recursion becomes
  \begin{displaymath}
    f_{_{\frac{r_1}{s_1} \oplus \frac{r_2}{s_2} }}(z) = 8 \cos^2(\theta) - f_{_{\frac{r_1}{s_1}}}(z) \; f_{_{\frac{r_2}{s_2}}}(z) -f_{_{\frac{r_1}{s_1} \ominus \frac{r_2}{s_2} }}(z)  \\
  \end{displaymath}
  with base cases $ f_{0/1}(z) = 2-z $ and $ f_{1/1}(z) = z+2\cos(2\theta) $.

  One expects that the Chebyshev polynomials should appear,  and indeed  it is not difficult to see that
  \begin{displaymath}
    \Phi_{1/n}^{\infty,\infty}(z)=\frac{2}{z-4}\left( T_n\left(\frac{z}{2}-1\right)z-4\right)
  \end{displaymath}
  where $T_n$ is the $n$th Chebyshev polynomial of the first kind, see \cref{sec:properties}. More generally, for other values of $\theta$,
  \begin{displaymath}
    f_{1/n}(z) = \frac{1}{z-4}\left(2 T_n\left(\frac{z}{2}-1\right)(z-4\sin^2(\theta))-8\cos^2(\theta)\right).
  \end{displaymath}
\end{ex}

\subsection{Structure of the paper}
The remainder of this introduction describes some the applications of our results, including some geometric context; it is not necessary for the reader who is only
interested in the abstract combinatorics of the polynomials to read this, and they may skip to \cref{sec:cutting} which gives
a combinatorial definition of the Farey words using cutting sequences; we take care to avoid requiring any knowledge of curve coding. In \cref{sec:farey}
we list some definitions and results about Farey sequences from classical number theory. Our
main result (\cref{thm:recursion}) is proved in \cref{sec:recursion}; in the following section, \cref{sec:commutators}, we use our results to compute the commutators of the Farey words
in a couple of different ways. The final part of the paper consists of two sections; in \cref{sec:properties}, we apply the work of Chesebro and his collaborators \autocite{chesebro19,chesebro20}
to show a formal analogy between the Farey polynomials and the Chebyshev polynomials. Then in \cref{sec:contin_frac} we discuss some applications to the approximation of irrational pleating
rays and cusps, and give some computational results which show that there are interesting connections with dynamical systems and number theory for future work to explore.

\subsection{Applications and connections to geometry and topology}\label{sec:context}
It is the following remarkable connection to three dimensional hyperbolic geometry and potential applications that motivates our study; we give only a very rough idea here, and precise definitions
can be found in our work \autocite{ems21}. Fix $a,b$ integers which are at least $2$ or infinity, not both $2$, and let $\mathcal{R}^{a,b}$ denote the closure of the quasiconformal deformation
space of faithful representations (up to conjugacy) of the free product group $\Z/a\Z * \Z/b\Z$ into ${\PSL}(2,\mathbb{C})$.  There is a natural embedding of $\mathcal{R}^{a,b}$ in the complex plane,
namely the set of $ z \in \C $ such that the group $ \Gamma_z^{a,b} $ is discrete and isomorphic to $ \Z/a\Z * \Z/b\Z $. In this guise, it forms the region of \cref{RileyCusped} (which shows the
case $ a = b = \infty $, though the cases for finite $ a $ and $ b $ look similar) laminated with smooth curves (along with the boundary of that set). The interior of this set is the $(a,b)$-Riley slice, $ \mc{R}^{a,b} $.

\begin{figure}
  \centering
  \scalebox{0.5}{\includegraphics[viewport=-30 330 600 900]{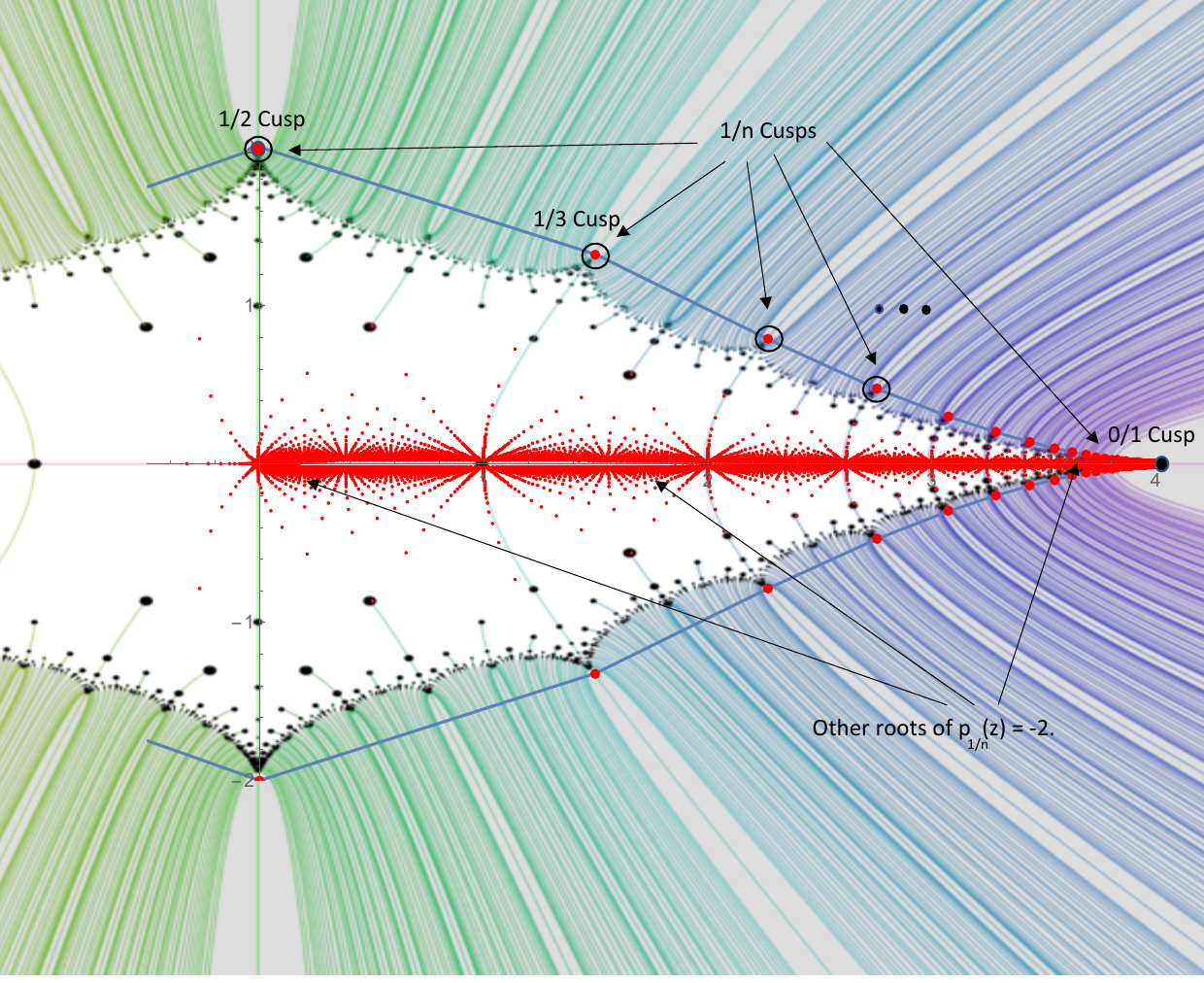}}\\
  \caption[Cusp groups on the boundary of the Riley slice]{Cusp groups on the boundary of the Riley slice $ \mc{R}^{\infty,\infty} $ (courtesy of Y. Yamashita). As $n\to\infty$ the $\frac{1}{n}$-cusp groups accumulate at the $(\infty,\infty,\infty)$-triangle group. The other solutions to $ \Phi^{\infty,\infty}_{1/n}(z)=-2$ are shown in red.\label{RileyCusped}}
\end{figure}

A detailed historical account of the Riley slice, together with background information for the non-expert and motivating applications (including a brief outline of earlier work
on the classification of two-generated arithmetic groups), can be found in \autocite{ems22M}.
Though it was first defined in the mid-20th century, recent work on the Riley slice follows on from a 1994 paper of Linda Keen and Caroline Series \autocite{keen94}
(with some corrections by Yohei Komori and Caroline Series \autocite{komoriseries98}). In this paper, they constructed a foliation of the Riley slice for two parabolic generators $ \mc{R} = \mc{R}^{\infty,\infty} $ via a two-step process:
\begin{enumerate}
  \item Firstly, a lamination of $ \mc{R} $ is defined, with leaves indexed by $ \Q/2\Z $. These leaves consist of certain branches of preimages of $ (-\infty,-2) $ under a family of
        polynomials $ \Phi_{p/q} $ ($p/q \in \Q$) and are called the \textbf{rational pleating rays}.
  \item Then, by a completion construction similar to the completion of $ \Q $ via Dedekind cuts, the lamination is extended to a foliation whose
        leaves are indexed by $ \R/2\Z $. The leaves adjoined at this step are called the \textbf{irrational pleating rays}.
\end{enumerate}
The polynomials $ \Phi_{p/q} $ are constructed as trace polynomials of certain words in the group  $ \Gamma_z $, which we call \textbf{Farey words}; these words
enumerate all but one of the non-boundary-parallel simple closed curves on the four-punctured sphere up to homotopy (the one that is missing is exactly the one bounding
a compression disc in the 3-manifold), and the distinguished branches on the preimages of $ (-\infty,-2) $ correspond to curves in the Riley slice along which the
length of the pleating locus of the convex curve core (which is homotopic in the 3-manifold to the $p/q$ curve on $ \Omega(\Gamma_z)/\Gamma_z $ when $ z $ lies on
the curve in the Riley slice corresponding to the $ p/q $ polynomial) changes in a particularly natural way; in particular, the algebraic limit of $ \Gamma $
 as $ z $ travels down the $ (p/q)$-pleating ray towards the preimage of $ -2 $ has surface quotient equal to a pair of three-times
punctured spheres, where the additional pair of punctures appears as the loxodromic word $ W_{p/q} $ is pinched to a parabolic word with trace $ -2 $.
The group corresponding to this limit is known as a \textbf{cusp group}. The set of points corresponding to cusp groups are dense in the boundary of the
slice by a result of Canary, Culler, Hersonsky, and Shalen \autocite{canary03} and the structure of the cusp points in the boundary was studied in detail by Wright \autocite{wright05}.

Generalising work of Keen, Komori, and Series \autocite{keen94,komoriseries98}, we establish the following analogue for groups with elliptic generators in \autocite{ems22M,elzenaar22,ems22c}.
\begin{thm}\label{thm1}
  The set $\mathcal{R}^{a,b}$ is the interior of the unbounded component of
  \begin{displaymath}
    \mathbb{C}\setminus \overline{\bigcup_{p/q\in[0,1]} \{z: \Phi^{a,b}_{p/q}(z)=-2\}}
  \end{displaymath}
\end{thm}

A theorem of Miyachi and Ohshika \autocite{ohshika10}  informs us that the complement of the Riley slice is a Jordan domain which is not a quasidisc, so one expects a complicated boundary.
A point $z\in \partial \mathcal{R}^{a,b}$ where $\Phi^{a,b}_{p/q}(z)=-2$ is called a \textbf{cusp}, and in fact the Chebyshev polynomial identity in \cref{ex:special_case} can be used to
precisely identify the shape of the most prominent cusp of the moduli space $\mathcal{R}^{a,b}$. Miyachi \autocite{miyachi03} states that there is a ``universal shape'' to the all cusps
in $\mc{R}^{\infty,\infty}$ of the form $y^2=t x^3$,  with $t$ depending on the particular cusp,  and this persists for $\mathcal{R}^{a,b}$ (work in progress based on the current paper).
For example, for the Riley slice at the $0/1$-cusp at $4$ we have $t=\frac{1}{\pi^2}$.

A point to \cref{thm1}, and evidenced by \cref{ex:special_case} where we allowed $ \theta $ to range over values which do not correspond to orbifolds uniformised by discrete groups,
is that the Farey polynomials have allowed us to abstract away the very challenging questions concerning the discreteness of Kleinian groups and replace them with the study of root sets of
recursively defined polynomials to obtain information about the geometry of deformation spaces. These formulae may be used to generate high-resolution pictures of the deformation spaces far
more quickly than working directly from the definition of cutting sequences as it bypasses slow matrix multiplications; a working software implementation of this in \texttt{Python} may be found
online \autocite{Elzenaar_bella}.

We wish to make a few more observations about these Farey polynomials and their geometric interpretations.  The first is that the Farey polynomials are ``trace polynomials''.  The simple closed
curves on the four times punctured sphere which separate one pair of points from another are enumerated by a rational slope $p/q\in [0,1]$.  Each such slope is associated with a so-called \textbf{Farey word}
$ W_{p/q} $ in the fundamental group of the four times punctured sphere which represents the $p/q $ simple closed curve. As noted above, this was used by Keen and Series \autocite{keen94}
in this context and is carefully explained in \cref{sec:cutting} below via cut sequences. Relatedly, these rational slopes and Farey words also enumerate the two-bridge knots and links.
Riley \autocite{riley72} proved that the fundamental group of any hyperbolic two-bridge link has a faithful presentation on the two parabolic generators
\begin{displaymath}
 \begin{bmatrix} 1&1 \\ 0&1  \end{bmatrix} \mbox{ and } \begin{bmatrix}1 & 0 \\ z & 1  \end{bmatrix},
\end{displaymath}
for some $ z \in \C $; that is, these fundamental groups are of the form $ \Gamma^{\infty,\infty}_{z} $ for some $ z $. Farey words can be viewed as relators in the fundamental groups
of the $3$-manifold knot or link complement $\mathbb{S}^3\setminus K$ when the generators are chosen as meridians, and in fact this can be taken as another definition of a Farey
word \autocite{riley72}. In particular when $ z $ is such that $ \Gamma^{\infty,\infty}_z $ is a knot group the word $ W_{p/q} $ is the identiy, hence for the knot or link of slope $p/q$
we have $\Phi^{\infty,\infty}_{p/q}(z)=2$. One might ask what geometric interpretations exist when the parameters $ a, b $ are non-integral.

Far more is true due to the classification theorem \autocite{akiyoshi2020classification} of all discrete groups generated by two parabolics. A result from that paper which illustrates the connection of this classification
to the Farey words which we discuss in this paper is the following:
\begin{thm}\label{fareythm2}
  If $ \Gamma = \Gamma^{\infty,\infty}_z $ is discrete then either $\Gamma$ is free or there is a Farey word $W_{p/q}$ and integers $r$ and $s$ so that
  \begin{displaymath}
  \Phi^{\infty,\infty}_{p/q}(z_0) = 2\cos \frac{r\pi}{s} \in [-2,2].
  \end{displaymath}
\end{thm}
That is, all the discrete groups of the form $ \Gamma_{z}^{\infty,\infty} $ which lie in the Riley slice exterior can be detected by the Farey words as well. In upcoming work, Chesebro,
Schillewaert, and Martin will use similar techniques to prove an analogue of \cref{fareythm2} for groups generated by pairs of elliptic transformations, and so the results of the current
paper apply also to the enumeration of all discrete groups on two elliptic generators. The case of hyperbolic generators (which correspond to genus two surfaces) is much harder. Results
of Thurston and others on the combinatorics of surfaces suggest that there will be analogues of the Farey words in this case, but they will arise in systems of three (so the Farey
polynomials, which are algebraic maps into $ \C $, will become algebraic maps into $ \C^3 $). Our experiments lead us to believe that there are
challenges in this direction that will require new ideas compared to the ideas in this paper, and that the study of the genus two case will be interesting to undertake. As a measure of complexity, in the current
paper we see that the Farey polynomials are strongly related to the group of orientation-preserving mapping classes of the four-punctured sphere. It is well-known that this group embeds into $ \SL(2,\C) $.
The corresponding group for the genus two surface is also linear, but of dimension thirty-two times larger \autocite{bigelow00}.

\subsection{Conclusion}
Taken altogether, these observations concerning the Farey polynomials yield a practical and effective approach to computing in these quasiconformal deformation spaces, see as illustrated in \cref{deformationspaces}. We hope to ultimately describe effective algorithms (perhaps partial algorithms) to solve the following decision problem for rank-two subgroups of $\SL(2,\C)$:

\begin{problem}
  Given two matrices in $\SL(2,\C)$, decide whether the group $G$ they generate is discrete. If $G$ is discrete decide whether it is free or not, and if $G$ is not free effectively produce a relator.
\end{problem}

The combination of the current paper with our paper \autocite{ems21} goes some way towards solving this problem.

\begin{figure}
\centering
\includegraphics[width=0.75\textwidth]{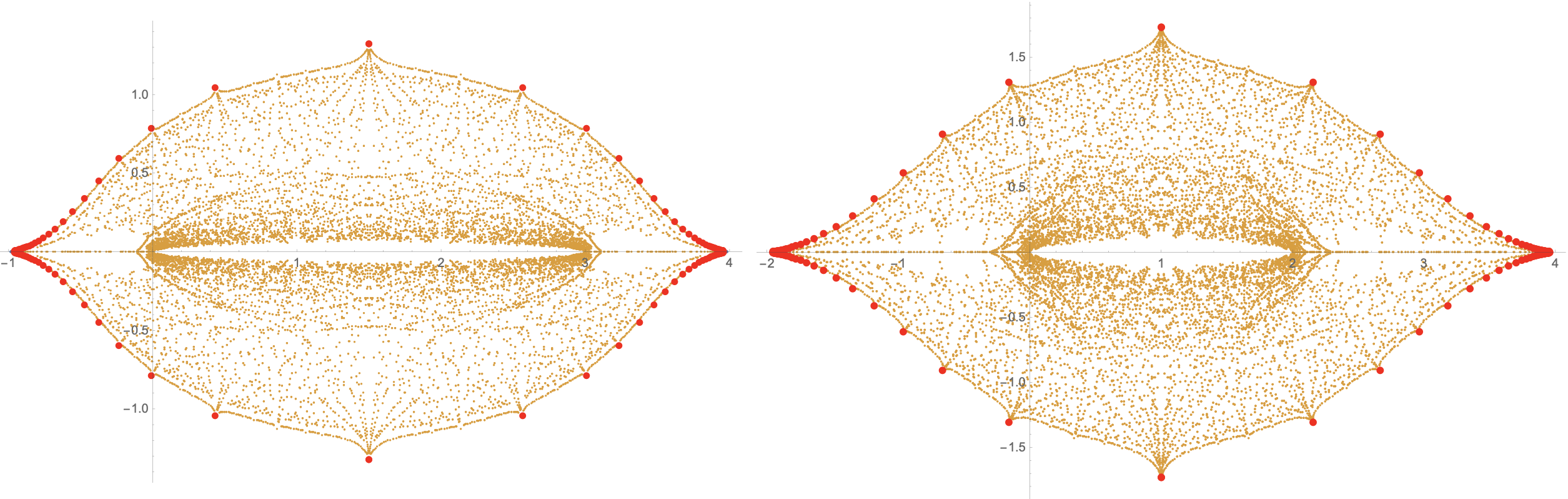}
\caption[Two rapidly computed deformation space complements.]{Two rapidly computed deformation space complements. Left: the quasiconformal deformation space of Kleinian groups generated by two elements of order $3$.  Right: the quasiconformal deformation space of Kleinian groups generated by two elements of order $4$.
\label{deformationspaces}}
\end{figure}

\subsection{Acknowledgements}
We thank the referee for the numerous good suggestions which have improved this paper.
Much of this paper appeared as part of A.E.'s MSc thesis \autocite{elzenaar22}. We would like to thank the organisers of the Groups and Geometry workshop on Waiheke Island, New Zealand, in December, 2021
for the opportunity to speak, and the attendees of that conference for their interesting questions and helpful comments. A.E. also thanks Thu Hien Nguyen for hosting him at the Mathematisches
Forschungsinstitut Oberwolfach during June 2022.

\section{Cutting sequences and Farey words}\label{sec:cutting}
We define the Farey word\footnote{We name these words and the related polynomials after John Farey Sr. as they are closely related to the so-called Farey sequences of rational numbers which we
will discuss briefly later in this paper; with regard to this attribution, we quote from the historical notes to Chapter III of Hardy and Wright \autocite[36--37]{hardywright}: ``The history of
`Farey series' is very curious... [their properties] seem to have been stated and proved first by Haros in 1802... Farey did not publish anything on the subject until 1816. [...] Mathematicians
generally have followed Cauchy's example in attributing the results to Farey, and the series will no doubt continue to bear his name. Farey has a notice of twenty lines in the \textit{Dictionary
of national biography} where he is described as a geologist. As a geologist he is forgotten, and his biographer does not mention the one thing in his life which survives.''} of slope $ p/q $ via
cutting sequences: this is essentially an interpretation of Dehn's classical algorithm for curves on surfaces \autocite[Chapter 6]{stillwell} in the language of symbolic dynamics and mapping class
groups by Birman and Series \autocite{birman-series}, and for the concrete case of interest to us we follow \autocite{seriesESH} and \autocite[\S 2.3]{keen94}.

Since we want to give a purely abstract version of the theory, we work in quite a formal way and do not use the geometric language except in some remarks; for the interested reader,
we describe the geometric motivation briefly now. The reader who is interested in the combinatorics only may skip directly to \cref{sec:formal_farey} where we give
our formalised definition of the Farey words that is the basis for the combinatorial theory in the remainder of the paper.

Consider a group $ \Gamma_z $ such that $ \Omega(\Gamma_z)/\Gamma_z $ is a four-times punctured sphere. Cut along two disjoint arcs on the surface, one for each generator of $ \Gamma $ joining the two images of the fixed point of the generator (i.e. horocyclic arcs); this induces a lift of the quotient surface to $ \Omega(\Gamma) $ which is bounded by four curves (not circular arcs in general) paired by the group generators and therefore we obtain a
tiling of $ \Omega{\Gamma} $ where each edge has two labels, one from each tile incident to the edge. We cut again along an arc
joining the fixed points of the two generators; the resulting object is a hexagon with sides identified in pairs, and this object tiles $ \R^2 $ via the edge pairings as shown in \cref{fig:fundamental_domain};
simple closed curves on the four-times punctured sphere lift to lines on $ \R^2 $ that miss the vertices of these hexagons, and the word representing the curve is read off by listing the edges that the line
crosses (always listing the label on a consistent side of each edge). The result is called a \textbf{cutting sequence} for the curve. There is a similar picture in the case of elliptic generators, but if
we restrict to simple closed curves up to orbifold homotopy we get the same result.

This procedure actually induces a bijection between elements of $ \hat{\Q} $ and non-boundary-parallel simple geodesics of minimal length on the four-punctured sphere $ S_{0,4} $ (given any particular choice of
hyperbolic structure): there is a unique such geodesic in the free homotopy class of the projection from $ \R^2 $ of a line of slope $ p/q $ (where the line is chosen so that it does not meet the lattice $ \Z^2 $;
the choice of line does not matter). The vertical line (the one with slope $ 1/0 $) does not meet any of the labelled edges, and so is given the `empty' cutting sequence that corresponds to the trivial word. This is the geodesic in
the unique free homotopy class of non-boundary-parallel curves on $ \Omega(\Gamma)/\Gamma $ that are isotopic to the identity in $ \H^3/\Gamma $.

\begin{figure}
  \centering
  \includegraphics[width=\textwidth]{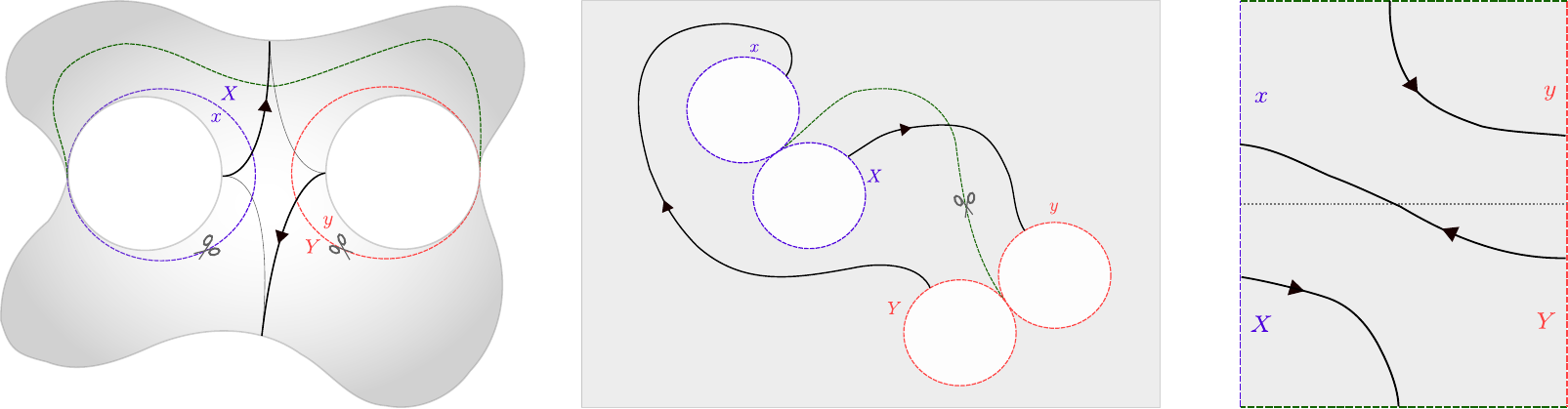}
  \caption{Geodesic coding of free homotopy classes of simple closed curves on the four-punctured sphere.\label{fig:fundamental_domain}}
\end{figure}

\subsection{Formalised definition of the Farey words}\label{sec:formal_farey}
Fix a rational number $ p/q $. Consider the marked tiling of $ \R^2 $ shown in \cref{fig:tiling}, and let $ L_{p/q} $ be the line through $ (0,0) $ of slope $ p/q $.
The \textbf{Farey word of slope $ p/q $},  $ W_{p/q} $, is the word of length $ 2q $ such that the $k$th letter is the label on the right-hand side of the $k$th vertical line segment
crossed by $ L_{p/q} $ (i.e. the label to the right of the point $ (p/q) k $). If $ (p/q) k $ is a point with integer coordinates then this definition is ambiguous and by convention we take the label
on the north-east side (we `nudge' the sloped line up around the lattice points). In other words, the symbols alternate between $ Y^{\pm 1} $ and $ X^{\pm 1} $, with the $k$th exponent given by
\begin{gather*}
  -1^{\ceil [(p/q)k]} \text{ if $k$ is odd (corresponding to $ Y$'s)},\\
  -1^{1+\ceil [(p/q)k]} \text{ if $k$ is even (corresponding to $X$'s)}
\end{gather*}
in which the ceiling is taken with the convention that $ \ceil n = n+1 $ for integral $ n $. By convention, we also define $ W_{1/0} \coloneq 1 $.

\begin{figure}
  \centering
  \includegraphics[width=0.5\textwidth]{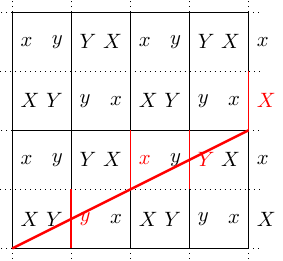}
  \caption[Another view of the cutting sequence for $ W_{1/2}$.]{The cutting sequence of the $1/2$ Farey word.\label{fig:tiling}}
\end{figure}

\begin{ex}
  As an example of the construction process, from \cref{fig:tiling} we can read off that $ W_{1/2} = yxYX $; we include a list of the Farey words for small $q$ as \cref{tab:words}.
\end{ex}

\begin{table}
  \centering
  \caption{Farey words $ W_{p/q} $ and $ W_{-p/q} $ for small $ q $. The vertical bars splitting the words are explained in \cref{lem:top_bottom_reflection}.\label{tab:words}}
  \begin{tabular}{lll}
    \toprule
    $p/q$ & $ W_{p/q} $ & $ W_{-p/q} $\\%
    \midrule
    $0/1$ & $|yX$ & $yX$\\
    $1/1$ & $|YX$ & $YX$\\
    $1/3$ & $yX|YxYX$ & $YxYXyX$\\
    $2/3$ & $yxyX|YX$ & $YXyxyX$\\
    $1/5$ & $yXyX|YxYxYX$ & $YxYxYXyXyX$\\
    $2/5$ & $yX|YxyXyxYX$ & $YxyXyxYXyX$\\
    $3/5$ & $yxYXYxyX|YX$ & $YXyxYXYxyX$\\
    $4/5$ & $yxyxyX|YXYX$ & $YXYXyxyxyX$\\
    $1/7$ & $yXyXyX|YxYxYxYX$ & $YxYxYxYXyXyXyX$\\
    $2/7$ & $yXyxYxyXyX|YxYX$ & $YxYXyXyxYxyXyX$\\
    $3/7$ & $yX|YxyXYxYXyxYX$ & $YxyXYxYXyxYXyX$\\
    $4/7$ & $yxYXyxyXYxyX|YX$ & $YXyxYXyxyXYxyX$\\
    $5/7$ & $yxyX|YXYxyxYXYX$ & $YXYxyxYXYXyxyX$\\
    $6/7$ & $yxyxyxyX|YXYXYX$ & $YXYXYXyxyxyxyX$\\
    $1/9$ & $yXyXyXyX|YxYxYxYxYX$ & $YxYxYxYxYXyXyXyXyX$\\
    $2/9$ & $yXyX|YxYxyXyXyxYxYX$ & $YxYxyXyXyxYxYXyXyX$\\
    $4/9$ & $yX|YxyXYxyXyxYXyxYX$ & $YxyXYxyXyxYXyxYXyX$\\
    $5/9$ & $yxYXyxYXYxyXYxyX|YX$ & $YXyxYXyxYXYxyXYxyX$\\
    $7/9$ & $yxyxYXYXYxyxyX|YXYX$ & $YXYXyxyxYXYXYxyxyX$\\
    $8/9$ & $yxyxyxyxyX|YXYXYXYX$ & $YXYXYXYXyxyxyxyxyX$\\%
    \midrule
    $1/2$ & $y|xYX$ & $YxyX$\\
    $1/4$ & $yXy|xYxYX$ & $YxYxyXyX$\\
    $3/4$ & $yxy|xYXYX$ & $YXYxyxyX$\\
    $1/6$ & $yXyXy|xYxYxYX$ & $YxYxYxyXyXyX$\\
    $5/6$ & $yxyxy|xYXYXYX$ & $YXYXYxyxyxyX$\\
    $1/8$ & $yXyXyXy|xYxYxYxYX$ & $YxYxYxYxyXyXyXyX$\\
    $3/8$ & $yXYxYXy|xYxyXyxYX$ & $YxyXyxYxyXYxYXyX$\\
    $5/8$ & $yxYXYxy|xYXyxyXYX$ & $YXyxyXYxyxYXYxyX$\\
    $7/8$ & $yxyxyxy|xYXYXYXYX$ & $YXYXYXYxyxyxyxyX$\\%
    \bottomrule
  \end{tabular}
\end{table}

\begin{rem}
  Note that the number of reduced words of length at most $n$ in $\langle X,Y\rangle$ grows like $4\cdot3^{n-1}$. The word length of $W_{p/q}$ is $2q$. Hardy and Wright prove \autocite[Theorem 330 p.268]{hardywright}
  that the length of the Farey sequence of order $2q$, that is the sequence of all reduced fractions between $0$ and $1$ with denominator at most $2q$, is asymptotic to $\frac{3}{\pi^2}(2q)^2$.  Thus given a
  group $ \Gamma_z $ which is discrete but not free, \cref{fareythm2} is a powerful tool for finding relators of length $q$ or less.  We need only examine the approximately $q^2$ values $\Phi^{\infty,\infty}_{p/q}(z_0)$
  of the (roughly) $4\cdot3^{2q-1}$ possibilities for the relator.  Note that it is not claimed this is a complete presentation (and sometimes it is not) but once we know $ W_{p/q}$ is of finite order then
  this is often enough to identify the group as a Dehn surgery on a knot or link complement. This has proved a useful tool in the enumeration of arithmetic groups of $\PSL(2,{\mathbb C})$ generated by two
  elements. Conjecturally there are only finitely many such groups,  and it is known that there are only finitely many such groups of $\PSL(2,{\mathbb C}) $ generated by two parabolic or elliptic generators
  \autocite{maclachlan99}. For instance the groups generated by two parabolic elements are known, and there are only four of them \autocite{gehring98}. All of these groups are knot or link groups  and lie
  in the exterior of the  Riley slice.
\end{rem}

\subsection{Symmetries of the Farey words}

There are various symmetries visible in the Farey words; for instance, they are alternating products of $ X^{\pm 1} $ and $ Y^{\pm 1} $ which always end in $ X $ (this is
obvious from the definition). We study first some periodicity in $ p/q $.
\begin{lem}\label{lem:periodicity_symmetries}
  Let $ p/q \in \Q $ be non-negative and given in lowest terms.
  \begin{enumerate}
    \item If $ p = nq + p' $, then in case $ n $ is even $ W_{p/q} = W_{p'/q} $ and in case $ n $ is odd $ W_{p/q} $ is obtained from $ W_{p'/q} $ by replacing $ Y \leftrightarrow Y^{-1} $.
    \item $ W_{-p/q} $ is obtained from $ W_{p/q} $ by the following process: swap the sign of the exponent of every letter except for the $ q$th and $ 2q$th letters.
    \item If $ p = nq - p' $, then in case $ n $ is even $ W_{p/q} $ is obtained from $ W_{p'/q} $ by replacing $ X \leftrightarrow X^{-1} $ and $ Y \leftrightarrow Y^{-1} $ in every position
          except the $q$th and $ 2q$th position and in case $ n $ is odd $ W_{p/q} $ is obtained from $ W_{p'/q} $ by replacing $ X \leftrightarrow X^{-1} $ in every position except the $q$th and $2q$th position,
          as well as swapping $ Y \leftrightarrow Y^{-1} $ if it occurs in $q$th or the $2q$th position.
  \end{enumerate}
\end{lem}
\begin{proof}
  The first part is obtained using the fact that translating the intersection points of the cutting sequence vertically by even integers does not change the labels while translating
  by odd integers swaps the labels.

  To prove the second part, for all $ k $ such that $ kp/q $ is non-integral, we have $ \ceil -kp/q = -\floor kp/q = 1-\ceil (kp/q) $.
  In particular, the parity of the ceiling is swapped and so the sign of the exponent is swapped. For $ k $ such that $ p/q $ is integral (i.e. $ k $ is $ q $ or $ 2q $), due to our
  convention on the ceiling function, we have $ \ceil (-kp/q) = 1-kp/q $
  and $ \ceil (kp/q) = kp/q + 1 $, so the parities of $ \ceil (-kp/q) $ and $ \ceil (kp/q) $ are equal and the exponent of the corresponding letter is unchanged.

  The third part then follows by directly combining the first and second parts.
\end{proof}
By this lemma the words $ W_{p/q} $ for $ p/q \geq 0 $ are periodic in $ p/q $ with period $ 2 $.

\begin{lem}\label{lem:top_bottom_reflection}
  Let $ p/q \in (0,1) $ and let $ n \in \Z $. Then the word $ W_{n - p/q} $ is a cyclic permutation of $ W_{n+p/q} $ (if $ q $ is odd)
  or of $ W_{n+p/q}^{-1} $ (if $ q $ is even).
\end{lem}
The vertical bars in \cref{tab:words} show the necessary cyclic permutation: the words are rotated left until the vertical bar is at the front. The point of this
lemma is that $ W_{\pm p/q} $ represent the same free homotopy class of unoriented curves on $ S_{0,4} $.
\begin{proof}
  By part (1) of \cref{lem:periodicity_symmetries} it suffices to prove the result when $ n = 0 $.

  If $ q $ is odd, then let $ m $ be the unique even element of $ \{0,...,2q-1\} $ which solves the equation $ mp \equiv -1 \mod q $ (there are
  exactly two solutions to this equation, one even and one odd). Let $ J $ be the initial word of $ W_{p/q} $ of length $ m $. Then $ V = J^{-1} W_{p/q} J $
  is a cyclic permutation of $ W_{p/q} $, and since we know the parity of $ m $ we know it begins with some power of $ Y $. Note that $ mp/q = (lq - 1)/q = l-1/q $
  for some $ l \in \Z $; since $ m $ is even, $ l $ must be odd, and therefore the $ m$th letter of $ W_{p/q} $ is $ X $ and the $ (m+1)$th letter
  of $ W_{p/q} $ (the first letter of $ V $) is $ Y $. In general, comparing $ (m+k)p/q $ to $ kp/q $ gives us $ (m+k)p/q - kp/q = l - 1/q $; since $l$ is odd,
  increasing a height\footnote{We will always use the word `height' in the context of cutting sequences to mean the $ y$-ordinate at a particular integral $ x$-value.}
  by $ l $ always swaps the label $ X \leftarrow x $ or $ Y \leftarrow y $; subtracting $ 1/q $ does not change this unless the original height
  was $ kp/q $ for $ k \equiv 0 \mod q $. In other words, $ V $ is obtained   from $ W_{p/q} $ by swapping $ X \leftrightarrow x $ and $ Y \leftrightarrow y $ at
  every position except $ q $ and $ 2q $; by part (2) of \cref{lem:periodicity_symmetries} we have $ V = W_{-p/q} $.

  If $ q $ is even, then there are no even solutions to $ mp \equiv -1 \mod q $. In this case we let $ m = q-1 $. As above, let $ J $ be the $m$-initial
  segment of $ W_{p/q} $: we claim that $ J^{-1} W_{p/q} J = W_{-p/q}^{-1} $. By \cref{lem:periodicity_symmetries}, $ W_{-p/q}^{-1} $ is obtained from $ W_{p/q} $
  by swapping the exponents in the $ q$th and $ 2q$th positions and then reversing the word. We therefore want to show that the cutting sequence letter at $ (m + k)p/q $ is the
  same as that at $ (2q + 1 - k)p/q $ except when $ k-1 $ is a multiple of $ q $. Noting that everything is defined mod $ 2q $, we ask when $ (m+k)p/q $ and $ (1-k)p/q $ cut the
  same letter. Note, $ (m+k)p/q = (q-1+k) p/q = p + (k-1)p/q $.  By the same discussion as the proof of (2) of \cref{lem:periodicity_symmetries}, $ (k-1)p/q $ and $ (1-k)p/q $
  cut letters of opposite exponent iff $ (k-1)p/q $ is non-integral; shifting the height by the odd integer $ p $ always swaps the cutting sequence, hence the exponents swap
  twice giving no net change if $ k-1 $ is not a multiple of $ q $ and swap only once if $ k-1 $ is a multiple of $ q $.
\end{proof}

\begin{lem}\label{lem:wirtinger_symmetry}
  Let $ W_{p/q} $ be a Farey word; then the word consisting of the first $ 2q-1 $ letters of $ W_{p/q} $ is conjugate
  to $ X $ or $ Y $ according to whether the $ q$th letter of $ W_{p/q} $ is $ X^{\pm 1} $ or $ Y^{\pm 1} $ (i.e. according to
  whether $ q $ is even or odd respectively).
\end{lem}
\begin{proof}
  This identity comes from considering the rotational symmetry of the line of slope $ p/q $ about the point $ (q,p) $; it is clear from the symmetry of the picture
  that the first $ p-1 $ letters of $ W_{p/q} $ are obtained from the $ (p+1)$th to $ (2p-1)$th letters by reversing the order and swapping the case (imagine rotating the line
  by 180 degrees onto itself and observe the motion of the labelling).
\end{proof}
\begin{rem}
  An alternative interpretation of this, which originates with Riley \autocite{riley72}, is that one can write down a presentation for the $ q/p $ 2-bridge
  link group of the form $ \langle X, Y : W_{p/q} = 1 \rangle $
  by taking the 2-bridge diagram, letting $ X $ and $ Y $ represent the generators of the fundamental group corresponding to loops around the two bridges, and then writing down a conjugation relation
  of the form $ X = VXV^{-1} $ or $ Y = VXV^{-1} $ that comes from starting at $ X $, walking around the link building up conjugations as you walk under or over crossings, and stopping when you get either
  to the other bridge (in the case of a knot) or back to the bridge you started with (in the case of a two-component link); then $ W_{p/q} $ is a cyclic permutation of $ V XV^{-1} X^{-1} $ for a knot
  or $ V XV^{-1} Y^{-1} $ for a two-component link. This $ V $ is called the $ p/q $ \textbf{Riley word}.
\end{rem}

\subsection{Farey polynomials}

The \textbf{Farey polynomial of slope $ p/q $} is defined by $ \Phi_{p/q} \coloneq \tr W_{p/q} $; this is a polynomial in $ z $ of degree $ q $, with coefficients rational
functions of $ \alpha $ and $ \beta $. If we wish to emphasise the dependence on $ a $ and $ b $, we write $ \Phi^{a,b}_{p/q} $.

\begin{rem}\label{rem:riley}
  The Farey polynomials are not to be confused with the so-called \textbf{Riley polynomials} $ \Lambda_{p/q} $ defined by Riley \autocite[Proposition 1]{riley72} and studied e.g.
  by Chesebro \autocite{chesebro19}; there, $ \Lambda_{p/q} $ is the polynomial in $ z $ which is the top-left entry of the word consisting of the first $ q-1 $ letters
  of $ W_{p/q} $ (i.e. the conjugating word in \cref{lem:wirtinger_symmetry}). See also Remark~5.3.13 of \autocite{akiyoshi}.
\end{rem}

\begin{table}
  \centering
  \caption{Farey polynomials for small $ q $.\label{tab:fareys}}
  \begin{tabular}[h]{rl}
    \toprule
    $p/q$ & $\Phi^{a,b}_{p/q}(z)$\\\midrule
    $0/1$&$ \frac\alpha\beta + \frac\beta\alpha - z$\\
    $1/1$&$\alpha\beta + \frac{1}{\alpha\beta} +z$\\
    $1/2$&$2 + \left(\alpha\beta -\frac\alpha\beta - \frac\beta\alpha + \frac{1}{\alpha\beta}\right)z + z^2$\\
    $1/3$&$\begin{aligned}[t]\frac{1}{\alpha\beta} + \alpha\beta &+ \left( 3 -\frac{1}{\alpha^2} - \alpha^2 - \frac{1}{\beta^2} - \beta^2 + \frac{\alpha^2}{\beta^2} + \frac{\beta^2}{\alpha^2} \right)z\\&+ \left(\alpha\beta -2\frac\alpha\beta - 2\frac\beta\alpha + \frac{1}{\alpha\beta}\right)z^2 + z^3 \end{aligned}$\\
    $2/3$&$\begin{aligned}[t]\frac\alpha\beta + \frac\beta\alpha &+ \left(-3 + \alpha^2 + \frac{1}{\alpha^2}  - \frac{1}{\alpha^2 \beta^2} - \alpha^2 \beta^2 + \beta^2 + \frac{1}{\beta^2}\right)z\\&+ \left(-2 \alpha \beta -\frac{2}{\alpha\beta} + \frac\alpha\beta + \frac\beta\alpha\right)z^2 - z^3\end{aligned}$\\
    $1/4$&$\begin{aligned}[t]2&+\left(\frac{\alpha}{\beta^3} - \frac{\alpha^3}{\beta^3} +  \frac{2}{\alpha\beta} - 3\frac\alpha\beta + \frac{\alpha^3}{\beta} + \frac{\beta}{\alpha^3} - 3\frac\beta\alpha + 2 \alpha \beta - \frac{\beta^3}{\alpha^3} + \frac{\beta^3}{\alpha}\right)z\\&+ \left(6 - \frac{2}{\alpha^2} - 2 \alpha^2 - \frac{2}{\beta^2} + 3 \frac{\alpha^2}{\beta^2} - 2 \beta^2 + 3\frac{\beta^2}{\alpha^2}\right)z^2\\&+ \left(\frac{1}{\alpha \beta} - 3\frac\alpha\beta - 3\frac\beta\alpha + \alpha \beta\right)z^3 + z^4\end{aligned}$\\
    $3/4$&$\begin{aligned}[t]2&+\left(\frac{1}{\alpha^3 \beta^3} - \frac{1}{\alpha \beta^3} - \frac{1}{\alpha^3 \beta} + \frac{3}{\alpha \beta} - 2\frac{\alpha}{\beta} - 2\frac\beta\alpha +  3 \alpha \beta - \alpha^3 \beta - \alpha \beta^3 + \alpha^3 \beta^3\right)z\\&+\left(6 - \frac{2}{\alpha^2} - 2 \alpha^2 - \frac{2}{\beta^2} +     \frac{3}{\alpha^2 \beta^2} - 2 \beta^2 + 3 \alpha^2 \beta^2\right)z^2\\&+\left(\frac{3}{\alpha \beta} - \frac\alpha\beta - \frac\beta\alpha +     3 \alpha \beta\right)z^3 + z^4\end{aligned}$\\\bottomrule
  \end{tabular}
\end{table}

\begin{table}
  \centering
  \caption{Farey polynomials indexed by the Fibonacci fractions.\label{tab:fareysfib}}
  \begin{tabular}[h]{r>{\small}l}\toprule
    $\frac{\fib(q-1)}{\fib(q)}$ & $ \Phi^{\infty,\infty}_{\fib(q-1)/\fib(q)}(z) $ \\\midrule
    $0/1$ & $ 2-z $\\
    $1/1$ & $ 2+z $\\
    $1/2$ & $ 2+z^2 $\\
    $2/3$ & $ 2-z-2z^2-z^3$\\
    $3/5$ & $ 2+z+2z^2+3z^3+2z^4+z^5$\\
    $5/8$ & $ 2 + 4 z^4 + 8 z^5 + 8 z^6 + 4 z^7 + z^8 $\\
    $8/13$ & $ \begin{aligned}[t]2 &-z - 2 z^2 - 5 z^3 - 12 z^4 - 22 z^5 - 32 z^6 - 44 z^7 - 54 z^8\\&- 53 z^9 - 38 z^{10} - 19 z^{11} - 6 z^{12} - z^{13}\end{aligned}$\\
    $13/21$ & $ \begin{aligned}[t]2 &+ z + 2 z^2 + 7 z^3 + 14 z^4 + 31 z^5 + 64 z^6 + 124 z^7 + 214 z^8 \\
    &+339 z^9 + 498 z^{10} + 699 z^{11} + 936 z^{12} + 1148 z^{13} + 1216 z^{14}\\
    &+ 1064 z^{15} + 746 z^{16} + 409 z^{17} + 170 z^{18} + 51 z^{19} + 10 z^{20} + z^{21} \end{aligned}$\\
    $21/34$ & $ \begin{aligned}[t]2 &+ z^2 + 8 z^4 + 24 z^5 + 68 z^6 + 192 z^7 + 516 z^8 +1256 z^9 + 2834 z^{10}\\
    &+ 5912 z^{11} + 11460 z^{12} + 20816 z^{13} +  35598 z^{14} + 57248 z^{15} \\
    &+ 86446 z^{16} + 122560 z^{17} + 163199 z^{18} + 203952 z^{19} + 238564 z^{20}\\
    &+ 259704 z^{21} + 260686 z^{22} + 238320 z^{23} + 195694 z^{24} + 142328 z^{25}\\
    &+ 90451 z^{26} + 49552 z^{27} +  23058 z^{28} + 8952 z^{29} + 2831 z^{30} + 704 z^{31}\\
    &+ 130 z^{32} + 16 z^{33} + z^{34} \end{aligned}$\\
    $34/55$ & $ \begin{aligned}[t] 2 &-z - 4 z^2 - 10 z^3 - 34 z^4 - 103 z^5 - 286 z^6 - 791 z^7\\
    &- 2078 z^8 - 5221 z^9 - 12680 z^{10} - 29824 z^{11} -  67872 z^{12}\\
    &- 149896 z^{13} - 321800 z^{14} - 671896 z^{15} - 1364228 z^{16}\\
    &- 2692102 z^{17} - 5158232 z^{18} - 9587668 z^{19} - 17273444 z^{20}\\
    &- 30141702 z^{21} - 50903644 z^{22} - 83138942 z^{23} - 131230688 z^{24} \\
    &- 200056876 z^{25} - 294348624 z^{26} - 417663240 z^{27} -  571010576 z^{28} \\
    &- 751328456 z^{29} - 950188464 z^{30} - 1153232920 z^{31} - 1340813030 z^{32} \\
    &- 1490107333 z^{33} - 1578696308 z^{34} - 1589182962 z^{35} - 1513960786 z^{36} \\
    &- 1358696535 z^{37} - 1142850158 z^{38} - 896137319 z^{39} - 651440922 z^{40} \\
    &- 436582355 z^{41}- 268228504 z^{42} - 150207744 z^{43} - 76207672 z^{44} \\
    &- 34797892 z^{45} - 14193584 z^{46} - 5125756 z^{47} - 1621110 z^{48} \\
    &- 442809 z^{49} - 102556 z^{50} - 19630 z^{51} - 2990 z^{52}\\
    &- 341 z^{53} - 26 z^{54} - z^{55} \end{aligned} $\\\bottomrule
  \end{tabular}
\end{table}

\begin{ex}
  We list the Farey polynomials $\Phi^{a,b}_{p/q} $ with $ q \leq 4 $ in \cref{tab:fareys}. This illustrates some of the difficulty in studying these polynomials:
  they are clearly very symmetric, but quickly become too unwieldy to write explicitly and so actually guessing what the symmetries \emph{are} in general is hard.
  We also list the first few `Fibonacci' Farey polynomials $ \Phi^{\infty,\infty}_{\fib(q-1)/\fib(q)} $ (as usual, $ \fib(1) = 1 $, $ \fib(2) = 1 $,
  and $ \fib(n) \coloneq \fib(n-1)+\fib(n-2) $) in \cref{tab:fareysfib}.
\end{ex}

Once the structure of the Riley slice is known to be given by the Farey polynomials in the manner outlined in \cref{sec:context}, the four-fold symmetry of the Riley
slice can be deduced directly from the following symmetries of the Farey polynomials.

\begin{prp}
  Let $ p/q \in \Q $ and $ n \in \Z $. Let
  \begin{displaymath}
    c = \alpha\beta^{-1} + \beta\alpha^{-1} - \alpha\beta - \alpha^{-1}\beta^{-1} = 4\sin\frac{\pi}{a} \sin \frac{\pi}{b}.
  \end{displaymath}
  Then:
  \begin{enumerate}
    \item The Farey polynomials are symmetric around integer slopes: $ \Phi^{a,b}_{n-p/q}(z) = \Phi^{a,b}_{n+p/q}(z) $.
    \item The Farey polynomials are periodic of period $2$, and periodic of period $1$ up to a substitution: if $ n $ is odd then $ \Phi^{a,b}_{n+p/q}(z) = \Phi^{a,b}_{p/q}(c-z) $,
          and if $ n $ is even then $ \Phi^{a,b}_{n+p/q}(z) = \Phi^{a,b}_{p/q}(z) $.
  \end{enumerate}
\end{prp}
\begin{proof}
  Part (1) follows directly from  \cref{lem:top_bottom_reflection}.
  Part (2) follows directly from (1) of \cref{lem:periodicity_symmetries} when $ n $ is even. When $ n $ is odd,
  by the same lemma we see that $\Phi^{a,b}_{n+p/q}(z)$ is obtained from $ \Phi^{a,b}_{p/q}(z)$ by swapping $ Y $ and $ Y^{-1} $, and this is
  equivalent to swapping $ z \leftrightarrow -z $ and $ \beta \leftrightarrow \beta^{-1} $. By Lemma~3.5.1 of \autocite{maclauchlanreid}, $ \tr W_{p/q} $ is
  an integer polynomial in $ \tr X = \alpha + \alpha^{-1} $, $ \tr Y = \beta + \beta^{-1} $, and $ \tr XY = \alpha\beta + \alpha^{-1}\beta^{-1} + z $. Write
  $ P $ for this polynomial, so
  \begin{displaymath}
    \Phi^{a,b}_{p/q}(z) = P(\tr X, \tr Y, \tr XY) \text{ and } \Phi^{a,b}_{n+p/q}(z) = P(\tr X, \tr Y^{-1}, \tr XY^{-1}).
  \end{displaymath}
  Observe that $ \tr XY^{-1} = \alpha \beta^{-1} + \alpha^{-1}\beta - z $. Since $ \tr Y = \tr Y^{-1} $, we are simply comparing
  \begin{displaymath}
    P(\alpha+\alpha^{-1}, \beta+\beta^{-1}, \alpha\beta + \alpha^{-1}\beta^{-1} + z) \text{ and } P(\alpha+\alpha^{-1}, \beta+\beta^{-1}, \alpha \beta^{-1} + \alpha^{-1}\beta - z)
  \end{displaymath}
  and the two expressions are equal under the (involutive) change of variables
  \begin{displaymath}
    z \mapsto \left(\alpha\beta^{-1} + \beta\alpha^{-1} - \alpha\beta - \alpha^{-1}\beta^{-1} - z\right),
  \end{displaymath}
  i.e. $ z \mapsto c-z $.
\end{proof}

\section{Farey theory}\label{sec:farey}
Our recursion formula, like that for the Maskit slice in \autocite[283--285]{indras_pearls}, is a recursion down the Farey diagram. We recall some of the notation and ideas, which can be found for instance in  \autocite[\S 4.5]{knuthCM} or \autocite[Chapter III]{hardywright}.

The subgroup $ \Gamma_1 $ of $ \PSL(2,\Z) $ generated by the two matrices
\begin{displaymath}
  L = \begin{bmatrix} 1 & 0 \\ 1 & 1 \end{bmatrix} \quad\text{and}\quad R = \begin{bmatrix} 1 & 1 \\ 0 & 1 \end{bmatrix}.
\end{displaymath}
is the orientation-preserving half of the mapping class group of the four-times punctured sphere and hence are the natural group which
acts on the complex of curves on that surface (for more information on this perspective see the discussion in \autocite{akiyoshi}
and \autocite[Chapter 12]{burde}). The action of this group by matrix multiplication on vectors with coprime integer coordinates (i.e. elements of $ \Z^2 $
visible from the origin) is equivalent to the action of this group on $ \hat{\Q} $ via fractional linear transformations, and the system of coding
curves which we described in the previous section is exactly the method of assigning a bijection between curves and rational numbers
such that the two actions ``$\PSL(2,\Z) $ on $ \hat{\Q} $'' and ``$ \Mod(S_{0,4}) $ on curves'' are isomorphic.

We have more structure here. The group $ \SL(2,\Z) $ acts on itself by left multiplication. This action has an interesting preserved
set: if the columns of $ B \in \SL(2,\Z) $ are both visible from the origin of $ \Z^2 $, then its image $ AB $ under multiplication by
any matrix $ A \in \SL(2,\Z) $ has the same property. This preservation property descends to $ \PSL(2,\Z) $, so the group $ \Gamma_1 $ acts
on ordered pairs $ p/q, r/s \in \hat{\Q} $ such that $ \abs{ps-rq} = 1 $. Such a pair is called a pair of \textbf{Farey neighbours}. The
graph with vertices $ \hat{\Q} $ and edges joining Farey neighbours is called the \textbf{Farey diagram} and it induces a triangulation of $\H^2$
shown in \cref{fig:farey_tess}. We denote this complex by $ \mathcal{D} $.

\begin{figure}
  \centering
  \includegraphics[width=0.8\textwidth]{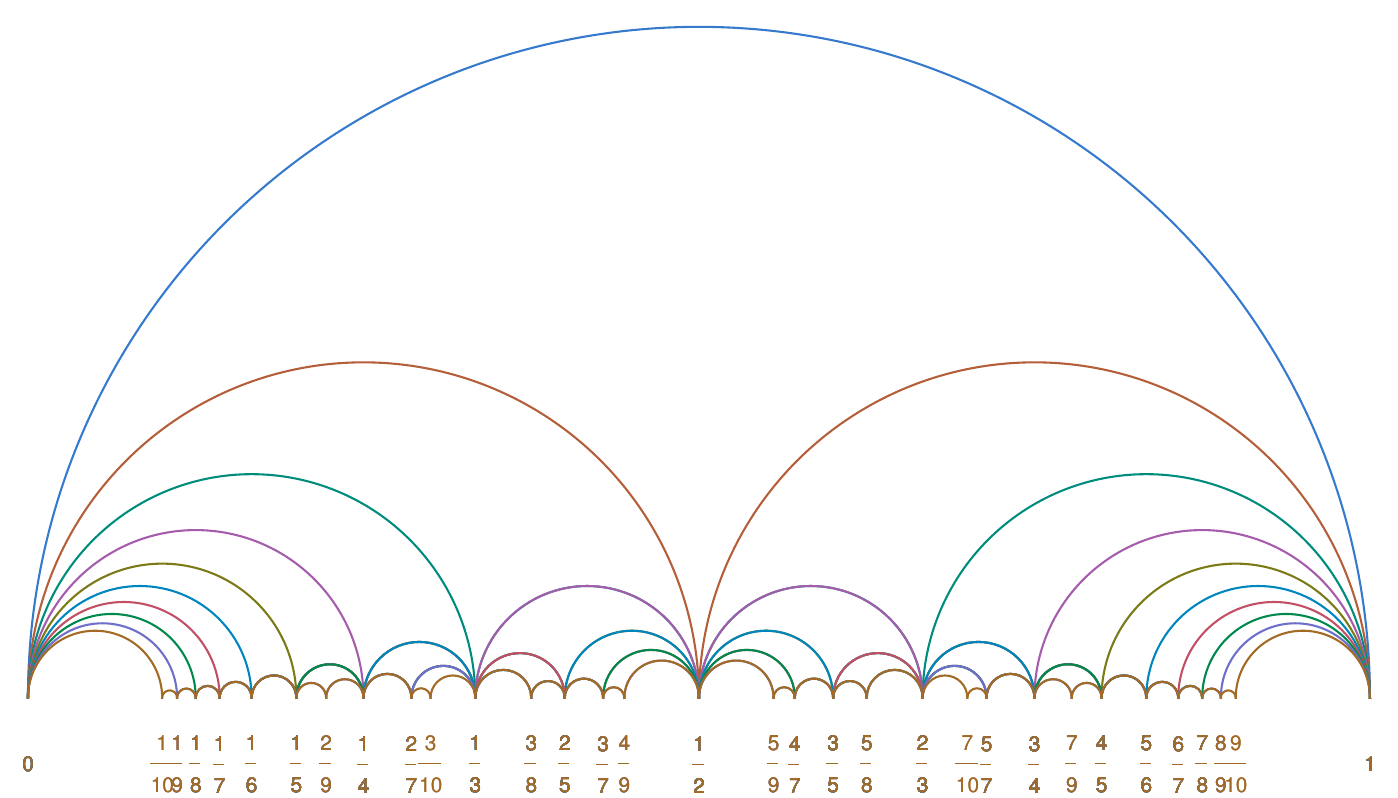}
  \caption{The subset of the Farey triangulation of $\H^2$ with vertices $ p/q \in [0,1]\inter\Q $ such that $ q \leq 10 $.\label{fig:farey_tess}}
\end{figure}

We now define a useful binary operation---this can be viewed either as a partial operation on $ \hat{\Q} $, or as a function from edges of the Farey diagram to $ \hat{\Q} $.
If $ p/q $ and $ r/s $ are Farey neighbours written in least terms, then we define
\begin{displaymath}
  \frac{p}{q} \oplus \frac{r}{s} \coloneq \frac{p+r}{q+s}.
\end{displaymath}
It is called the \textbf{mediant} or \textbf{Farey addition} operation. We have already mentioned that the Farey diagram induces a triangulation of $ \H^2 $, and in fact
for every edge joining two Farey neighbours we get a triangle consisting of those Farey neighbours together with their Farey sum; the triangles correspond to homology bases
for $ S_{0,4} $. The \textbf{Farey graph} is the digraph with vertices $ \hat{\Q} $ and directed edges from $ p/q $ and $ r/s $
to $ p/q \oplus r/s $ (\cref{fig:farey}). This is just the 1-skeleton of $ \mc{D} $ with an added choice of orientation for each edge.

\begin{figure}
  \centering
  \includegraphics[width=\textwidth]{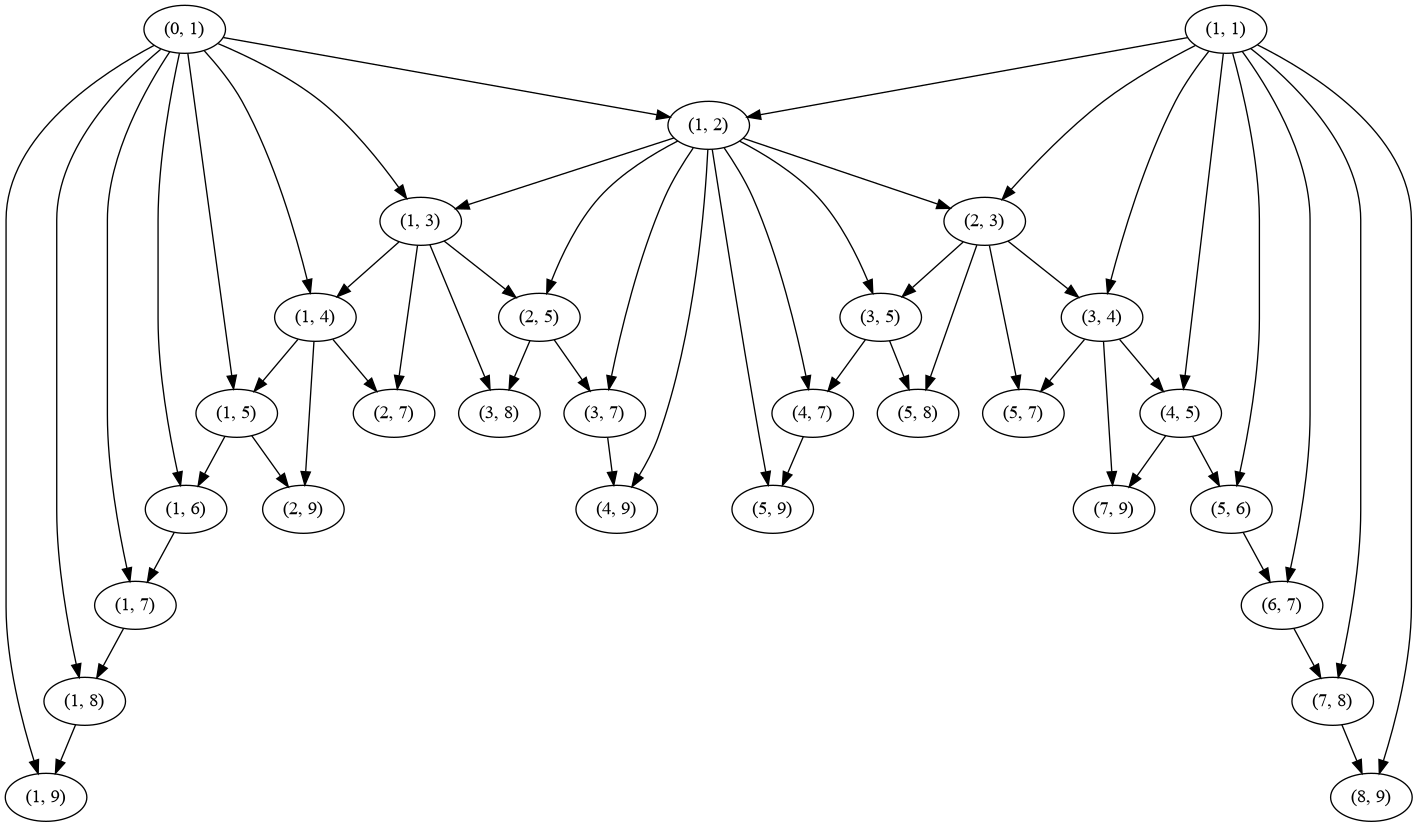}
  \caption{The Farey addition graph.\label{fig:farey}}
\end{figure}

Later on, we will need the following result which is a generalisation of the observation `the Farey sum of two rational numbers lies between them'.
\begin{lem}[{\autocite[\S 3.3]{hardywright}}]\label{lem:farey_closeness}
  If $ p/q $ and $ r/s $ are Farey neighbours with $ p/q < r/s $, then $ p/q < (p+r)/(q+s) < r/s $, and $ (p+r)/(q+s) $ is the unique fraction of minimal
  denominator between $ p/q $ and $ r/s $. More precisely, let $ u/v $ be any fraction in $ (p/q,r/s) $; then there exist two positive integers $ \lambda, \mu $
  such that
  \begin{displaymath}
    u = \lambda p + \mu r \text{ and } v = \lambda r + \mu s
  \end{displaymath}
  (so the minimal denominator is obtained when $ \lambda = \mu = 1 $). \qed
\end{lem}

It will be convenient, finally, to have the notation $ p/q \ominus r/s $ for the fraction $ (p-r)/(q-s) $; we will only use this when it is known that $ (p-r)/(q-s) $
and $ r/s $ are Farey neighbours (which is implied by neighbourliness of $ p/q $ and $ r/s $). If $ (p/q, r/s) $ is an edge in the Farey triangulation, then it forms
the boundary between the two triangles $ (p/q,p/q\oplus r/s,r/s) $ and $ (p/q,p/q\ominus r/s,r/s) $. Note that $ \ominus $ is commutative.

\section{A recursion formula to generate Farey polynomials}\label{sec:recursion}
In this section, we will give a recursion formula for the Farey polynomials. This recursion will be a recursion `down the Farey graph', in the sense that
its input will be the Farey polynomials at the vertices of a triangle $ (p/q,p/q \ominus r/s,r/s) $ and its output will be the Farey polynomial
at $ p/q \oplus r/s $:
\begin{displaymath}
  \vcenter{\hbox{\begin{tikzpicture}
    \node[shape=circle, fill=black, minimum size=5pt, inner sep=0pt, outer sep=0pt, label={above:$p/q \ominus r/s$}] (A) at (0,1/2) {};
    \node[shape=circle, fill=black, minimum size=5pt, inner sep=0pt, outer sep=0pt, label={below:$p/q \oplus r/s$}] (D) at (0,-1/2) {};
    \node[shape=circle, fill=black, minimum size=5pt, inner sep=0pt, outer sep=0pt, label={left:$p/q$}] (B) at (-1/2,0) {};
    \node[shape=circle, fill=black, minimum size=5pt, inner sep=0pt, outer sep=0pt, label={right:$r/s$}] (C) at (1/2,0) {};

    \draw[-stealth'] (A)--(B);
    \draw[-stealth'] (A)--(C);
    \draw[-stealth'] (B)--(C);
    \draw[-stealth',dotted] (B)--(D);
    \draw[-stealth',dotted] (C)--(D);
  \end{tikzpicture}}}
\end{displaymath}

We begin by finding a similar recurrence for the Farey words; we will then produce a recurrence for the polynomials using standard trace identities and elbow grease.
One might guess, for instance by analogy with the Maskit slice \autocite[277]{indras_pearls}, that $ W_{p/q} W_{r/s} = W_{p/q \oplus r/s} $ when $ p/q < r/s $ are Farey neighbours. It is easy to check whether
or not this is true:
\begin{gather*}
  W_{1/2} = yxYX, \qquad W_{1/1} = YX,\\
  W_{1/2} W_{1/1}  = yx{\color{red} Y}XYX, \text{ and}\\
  W_{1/2 \oplus 1/1} = W_{2/3} = yx{\color{red} y}XYX.
\end{gather*}
\begin{gather*}
  W_{1/3} = yXYxYX, \qquad W_{2/5} = yXYxyXyxYX,\\
  W_{1/3} W_{2/5} = yXYxYXy{\color{red} X}YxyXyxYX, \text{ and}\\
  W_{1/3 \oplus 2/5} = W_{3/8} = yXYxYXy{\color{red} x}YxyXyxYX.
\end{gather*}

These two examples suggest that our guess is almost correct, and that we need to make a correction in only one position.
\begin{lem}\label{lem:product_word}
  Let $ p/q $ and $ r/s $ be Farey neighbours with $ p/q < r/s $. Then $ W_{p/q \oplus r/s} $ is equal to:
  \begin{enumerate}
    \item $ W_{p/q} W_{r/s} $ with the sign of the $ (q+s)$th exponent swapped;
    \item $ W_{r/s} W_{p/q} $ with the sign of the $ (2s)$th exponent swapped.
  \end{enumerate}
\end{lem}
\begin{proof}
  We prove (1); the proof of (2) is done in the same way, and we briefly indicate the difference at the end.
  The situation is diagrammed in \cref{fig:concatenation} for convenience. To simplify notation, in this proof we write $ h(i) $ for
  the height $ (\frac{p}{q}\oplus \frac{r}{s}) i $. Observe that $ h(i) $ is integral only at $ i = 0 $ and $ i = 2q+2s $ (at both positions trivially the letters
  in $ W_{(p/q)\oplus(r/s)} $ and $ W_{p/q}W_{r/s} $ are identical) and at $ i = q+s $. The lemma will follow once we check that that at the positions $ i \not\in \{0,2q+2s\} $,
  \begin{displaymath}
    \text{if } 0 < i \leq 2q  \text{ then } (p/q) i < h(i) < \ceil (p/q) i
  \end{displaymath}
  and
  \begin{displaymath}
    \text{if } 0 < i < 2s \text{ then } (r/s)i + 2q < h(i + 2p) < \ceil [(r/s)i + 2p]:
  \end{displaymath}
  indeed, these inequalities show that at every integral horizontal distance the height of the line corresponding to $ W_{p/q \oplus r/s} $ is meeting the same vertical line segment
  as the line corresponding to $ W_{p/q} $ or $ W_{r/s} $, and so the letter chosen is the same except at $ i = q+s $ since at this position the height of the line of
  slope $ (p/q)\oplus(r/s) $, being integral, is rounded up to $ h(i)+1 $ while the height of the line of slope $ r/s $ is non-integral so is rounded up to the integer $ h(i) $.

  Observe now that the two displayed inequalities are equivalent to the following: there is no integer between $ (p/q)i $ and $ (p/q \oplus r/s)i $ (exclusive)
  if $ 0 < i \leq 2q $, and there is no integer between $ (r/s)i + 2q $ and $ h(i + 2p) $ if $ 0 < i < 2s $. But these follow from \cref{lem:farey_closeness}. Indeed,
  the lemma shows that no integer lies between $ p/q $ and $ (p/q) \oplus (r/s) $; suppose $ a/b $ is a rational between $ (p/q)i $ and $ h(i) $,
  then $ a = i(\lambda p + \mu (p+r)) $ and $ b = (\lambda q + \mu (q+s)) $ for some positive $ \lambda, \mu $; suppose $ a/b \in \Z $, so $ \lambda q + \mu (q+s) $ divides
  $ i(\lambda p + \mu (p+r)) $. By the case $ i = 1 $, $ \lambda p + \mu (p+r) $ and $ \lambda q + \mu (q+s) $ are coprime, so $ \lambda q + \mu (q+s) $ divides
  $ i $; but $ \lambda q + \mu (q+s) \geq 2q $. The case of $ (r/s)i + 2q $ and $ h(i + 2p) $ is proved in a similar way.

  The proof of part (2) differs only in the relative position of the `kink' between the concatenated $ p/q $ and $ r/s $ segments and the $ (p+r)/(q+s)$ segment: the
  kink is at position $ 2s $ horizontally and lies above the $ (p+r)/(q+s) $ line rather than below, so it is at the $(2s)$th position that the `modified ceiling'
  makes a nontrivial appearance.
\end{proof}

\begin{figure}
  \centering
  \includegraphics[width=\textwidth]{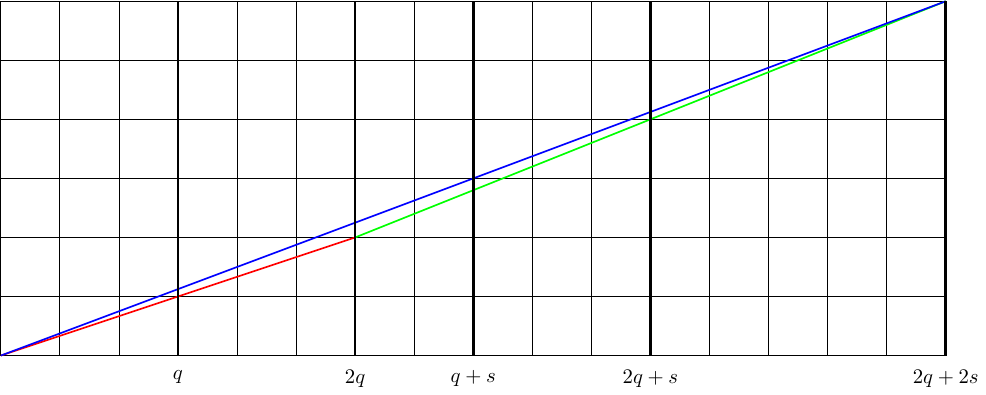}
  \caption{Farey addition versus word multiplication for $ W_{p/q} $ (red) and $ W_{r/s} $ (green).\label{fig:concatenation}}
\end{figure}

\Cref{lem:product_word} is not new. Part (1) appears as Propostion~4.2.4 of \autocite{zhang10}, and in the remark at the end of \S 4 of \autocite{keen94} Keen and Series indicate that they had
found an identity like part (2). However, the following consequences for the Farey polynomials are require different techniques to prove and have, to the best of our knowledge, not appeared before in the literature.

\begin{lem}[Product lemma]\label{lem:traceid}
  Let $ p/q $ and $ r/s $ be Farey neighbours with $ p/q < r/s $. Then the following trace identity holds:
  \begin{displaymath}
    \tr W_{p/q} W_{r/s} + \tr W_{p/q \oplus r/s} =
    \begin{cases}
      2 + \alpha^2 + \frac{1}{\alpha^2} & \text{if $ q+s $ is even,} \\
      \alpha\beta + \frac{\alpha}{\beta} + \frac{\beta}{\alpha} + \frac{1}{\alpha\beta} & \text{if $ q+s $ is odd.}
    \end{cases}
  \end{displaymath}
\end{lem}
\begin{proof}
  Trace is invariant under cyclic permutations, thus (applying \cref{lem:product_word}) we can write
  \begin{displaymath}
    \tr W_{p/q} W_{r/s} = \tr AB \text{ and } \tr W_{p/q \oplus r/s} = \tr AB^{-1},
  \end{displaymath}
  where $ B $ is the $ (q+s)$th letter of $ W_{p/q} W_{r/s} $ and $ A $ is the remainder of the word but with the final letters cycled to the front. Now
  we know that $ \tr AB = \tr A \tr B - \tr AB^{-1} $ (see the useful list of trace identities found in Section~3.4 of \autocite{maclauchlanreid}), so it
  suffices to check that $ \tr A \tr B = 2 + \alpha^2 + \frac{1}{\alpha^2} $ if $ q+s $ is even
  and $ \alpha\beta + \frac{\alpha}{\beta} + \frac{\beta}{\alpha} + \frac{1}{\alpha\beta} $ otherwise.

  \paragraph{Case I. $q+s$ is even.}
  Observe now that $ B $ is $ X^{\pm 1} $ if $ q+s $ is even; then $ \tr B = 2\Re \alpha $. The identity to show is therefore $ \tr A = ( 2 + \alpha^2 + \frac{1}{\alpha^2})/(2\Re\alpha) $;
  recalling that $ \abs{\alpha} = 1 $ and using the double angle formulae we have
  \begin{displaymath}
    \frac{2}{2\Re \alpha} + \frac{\alpha^2}{2\Re\alpha} + \frac{\alpha^{-2}}{2\Re\alpha} = \frac{1}{\cos\theta} + \left( \alpha - \frac{1}{2\cos\theta} \right) + \left( \overline\alpha - \frac{1}{2\cos\theta} \right) = 2\cos\theta
  \end{displaymath}
  where $ \theta = \arg \alpha $. Thus we actually just need to show $ \tr A = 2\cos\theta $.

  \paragraph{Case II. $q+s$ is odd.}
  In this case, $ B $ is $ Y^{\pm 1} $ and so $ \tr B = 2\Re\beta $; we therefore wish to show that $ \tr A = (\alpha\beta + \frac{\alpha}{\beta} + \frac{\beta}{\alpha} + \frac{1}{\alpha\beta})/(2\Re\beta) $;
  again using trigonometry we may simplify the right side,
  \begin{displaymath}
    \frac{\alpha\beta}{2\Re\beta} + \frac{\alpha/\beta}{2\Re\beta} + \frac{\beta/\alpha}{2\Re\beta} + \frac{1/(\alpha\beta)}{2\Re\beta}
      = 2\cos\theta
  \end{displaymath}
  and so again we need only show that $ \tr A = 2\cos\theta $ where $ \theta = \arg \alpha $.

  Both cases then reduce to the identity $ \tr A = \tr X $. It will be enough to show that $ A $ is conjugate to $ X $; by construction of $ A $, this is equivalent
  to showing that in $ W_{p/q \oplus r/s} $ the $ (q+s+1)$th to $ (2q+2s-1)$th letters are obtained from the first $ q+s-1 $ letters by reversing the order and swapping the case.
  But this is just \cref{lem:wirtinger_symmetry}.
\end{proof}

In the case that $ X $ and $ Y $ are parabolics and $ \alpha = \beta = 1 $, the two formulae unify to become:
\begin{displaymath}
  \tr W_{p/q} W_{r/s} = 4 - \tr W_{p/q \oplus r/s}.
\end{displaymath}

We may similarly prove the following `quotient lemma':
\begin{lem}[Quotient lemma]\label{lem:traceid2}
  Let $ p/q $ and $ r/s $ be Farey neighbours with $ p/q < r/s $. Then the following trace identity holds:
  \begin{displaymath}
    \tr W_{p/q} W_{r/s}^{-1} + \tr W_{p/q \ominus r/s} =
    \begin{cases}
      2 + \beta^2 + \frac{1}{\beta^2} & \text{if $ q-s $ is even,} \\
      \alpha\beta + \frac{\alpha}{\beta} + \frac{\beta}{\alpha} + \frac{1}{\alpha\beta} & \text{if $ q-s $ is odd.}
    \end{cases}
  \end{displaymath}
\end{lem}
\begin{proof}
  We begin by setting up notation. By \cref{lem:wirtinger_symmetry} we may write $ W_{p/q} = UAuX $ for some word $ U $ (recall our convention that, when
  writing words, $ u \coloneq U^{-1} $) with $ A = X^{\pm 1} $ if $ q $ is even and $ A = Y^{\pm 1} $ if $ q $ is odd; similarly, write $ W_{r/s} = VBvX $ for some word $ V $ and with $ B $
  equal to one of $ X^{\pm 1} $ or $ Y^{\pm 1} $. Then
  \begin{displaymath}
    W_{p/q} W_{r/s}^{-1} = UAuX xVbv = UAuVbv;
  \end{displaymath}
  by \cref{lem:product_word}, we have also that $ W_{r/s} W_{p/q \ominus r/s} $ is $ W_{p/q} $ with the sign of the exponent of the $q$th letter swapped; explicitly,
  \begin{displaymath}
    W_{p/q \ominus r/s} VBvX = UauX \implies W_{p/q \ominus r/s} = UauX xVbv = UauVbv.
  \end{displaymath}

  Our goal is therefore to compute $ \tr UAuVbv + \tr UauVbv $; performing a cyclic permutation again, this is equivalent to $ \tr A(uVbvU) + \tr a(uVbvU) $. In this form, this
  becomes
  \begin{displaymath}
     \tr A(uVbvU) + \tr a(uVbvU) = \tr A \tr uVbvU = \tr A \tr b.
  \end{displaymath}

  Consider now the cases for the product $ \tr A \tr b $:

  \begin{center}
    \begin{tabular}{ccc}
      \toprule
        & $ q $ odd & $ q $ even\\\midrule
      $s$ odd & $ \tr^2 Y $ & $ \tr X \tr Y $\\
      $s$ even & $ \tr X \tr Y $ & $ \tr^2 X $.\\\bottomrule
    \end{tabular}
  \end{center}

  If $ p/q, r/s $ are Farey neighbours then it is not possible for both $ q $ and $ s $ to be even since $ ps - rq \equiv 1 \pmod{2} $.
  Further, $ q-s $ is odd iff exactly one of $ p $ and $ q $ is odd, otherwise $ q-s $ is even. Thus we see that if $ q-s $ is even then
  \begin{displaymath}
    \tr W_{p/q} W_{r/s}^{-1} + \tr W_{p/q \ominus r/s} = \tr^2 Y = (\beta + 1/\beta)^2
  \end{displaymath}
  and if $ q-s $ is odd then
  \begin{displaymath}
    \tr W_{p/q} W_{r/s}^{-1} + \tr W_{p/q \ominus r/s} = \tr X \tr Y = (\alpha + 1/\alpha)(\beta + 1/\beta)
  \end{displaymath}
  which are the claimed formulae.
\end{proof}

Using the product and quotient lemmata, we can prove the desired recursion formula for the trace polynomials.

\begin{thm}[Recursion formulae]\label{thm:recursion}
  Let $ p/q $ and $ r/s $ be Farey neighbours. If $ q + s $ is even, then
  \begin{displaymath}
    \Phi_{p/q} \Phi_{r/s} + \Phi_{p/q \oplus r/s} + \Phi_{p/q \ominus r/s} = 4 + \frac{1}{\alpha^2} + \alpha^2 + \frac{1}{\beta^2} + \beta^2.
  \end{displaymath}
  Otherwise if $ q+s $ is odd, then
  \begin{displaymath}
    \Phi_{p/q} \Phi_{r/s} + \Phi_{p/q \oplus r/s} + \Phi_{p/q \ominus r/s} = 2\left( \alpha\beta + \frac{\alpha}{\beta} + \frac{\beta}{\alpha} + \frac{1}{\alpha\beta}\right).
  \end{displaymath}
\end{thm}
\begin{proof}
  Suppose $ q+s $ is even; then $ q-s $ is also even, so
  \begin{align*}
    &\Phi_{p/q} \Phi_{r/s} + \Phi_{(p+r)/(q+s)} + \Phi_{(p-r)/(q-s)}\\
      &\hspace*{5em} {}= \tr W_{p/q} \tr W_{r/s} + \tr W_{p/q \oplus r/s} + \tr W_{p/q \ominus r/s}\\
      &\hspace*{5em} {}= \tr W_{p/q} W_{r/s} + \tr W_{p/q} W_{r/s}^{-1} + \tr W_{p/q \oplus r/s} + \tr W_{p/q \ominus r/s}\\
      &\hspace*{5em} {}= 2 + \alpha^2 + \frac{1}{\alpha^2} + 2 + \beta^2 + \frac{1}{\beta^2}
  \end{align*}
  where in the final step we used \cref{lem:traceid} and \cref{lem:traceid2}. Similarly, when $ q-s $ is odd then $ q+s $ is also odd and
  \begin{align*}
    &\Phi_{p/q} \Phi_{r/s} + \Phi_{(p+r)/(q+s)} + \Phi_{(p-r)/(q-s)} \\
    &\hspace*{5em} {}= \tr W_{p/q} \tr W_{r/s} + \tr W_{p/q \oplus r/s} + \tr W_{p/q \ominus r/s}\\
    &\hspace*{5em} {}= \tr W_{p/q} W_{r/s} + \tr W_{p/q} W_{r/s}^{-1} + \tr W_{p/q \oplus r/s} + \tr W_{p/q \ominus r/s}\\
    &\hspace*{5em} {}= \alpha\beta + \frac{\alpha}{\beta} + \frac{\beta}{\alpha} + \frac{1}{\alpha\beta} + \alpha\beta + \frac{\alpha}{\beta} + \frac{\beta}{\alpha} + \frac{1}{\alpha\beta}
  \end{align*}
  as desired.
\end{proof}

\begin{cor}\label{cor:parabolic_identity}
  When both generators are parabolic, or in general whenever $ \alpha = \beta $, the recursion identity becomes
  \begin{displaymath}
    \Phi_{p/q} \Phi_{r/s} + \Phi_{(p+r)/(q+s)} + \Phi_{(p-r)/(q-s)} = 8
  \end{displaymath}
  with no parity dependence.\qed
\end{cor}
In the parabolic case this formula has been known to experts for a while but proofs of it go via either passing to a common
cover of the four-punctured sphere and the punctured torus and using the analogous identity for the latter \autocite{bowditch98,akiyoshi}, or by using directly some version of Bowditch's theory on Markoff triples \autocite{mpt15}, which we discuss further in the next section. A combinatorial
group theory approach as described above which allows the definition of the family of polynomials in a non-geometric abstract context does not appear in the literature.

The following corollary is useful when trying to draw moduli spaces of Riley groups on two elliptic generators, where it is possible that the boundary of the deformation
space is detected by nontrival substitutions of Farey words going parabolic while the Farey word in $ X $ and $ Y $ is still hyperbolic:
\begin{cor}[Recursion formula under substitution]\label{cor:substitution}
  Let $ \Gamma = \Gamma^{a,b}_z $ be endowed with the finite generating set $ \langle X, Y \rangle $ and let $ X', Y' \in \Gamma $. Suppose that there is an induced substitution
  homomorphism $ \sigma : \Gamma \to \Gamma $ obtained by extending the map defined on the generators by $ X \mapsto X' $ and $ Y \mapsto Y' $ to all words in $ X $ and $ Y $ (for
  instance, this is always true if $ \Gamma $ is free on $ X $ and $ Y $, or more generally if $ \Gamma \simeq \langle X \rangle * \langle Y \rangle $). For $ p/q \in \Q $ define
  a polynomial $ \sigma(\Phi) \in \C[z] $ to be
  \begin{displaymath}
    \sigma\left(\Phi(z)\right) = \tr \sigma\left(W_{p/q}\right)
  \end{displaymath}
  where $ \sigma $ is acting by substitution on the Farey word. Then for all Farey neighbours $ p/q $ and $ r/s $, if $ q+s $ is even then
  \begin{displaymath}
    \sigma\left(\Phi_{p/q}\right) \sigma\left(\Phi_{r/s}\right) + \sigma\left(\Phi_{p/q \oplus r/s}\right) + \sigma\left(\Phi_{p/q \ominus r/s}\right) = \tr^2 X' + \tr^2 Y'
  \end{displaymath}
  and if $ q+s $ is odd then
  \begin{displaymath}
    \sigma\left(\Phi_{p/q}\right) \sigma\left(\Phi_{r/s}\right) + \sigma\left(\Phi_{p/q \oplus r/s}\right) + \sigma\left(\Phi_{p/q \ominus r/s}\right) = 2\tr X' \tr Y'.
  \end{displaymath}
\end{cor}
\begin{proof}
  Observe that the proofs of \cref{lem:traceid} and \cref{lem:traceid2} are entirely combinatorial and one can simply
  substitute $ X \mapsto X' $ and $ Y \mapsto Y' $ to see that $ \tr W_{p/q} W_{r/s} + \tr W_{p/q \oplus r/s} $ is $ \tr^2 X $ (in the even case)
  or $ \tr X \tr Y $ (in the odd case), and $ \tr W_{p/q} W_{r/s}^{-1} + \tr W_{p/q \ominus r/s} $ is $ \tr^2 Y $ (even) or $ \tr X \tr Y $ (odd).
  The proof of the recursion formula \cref{thm:recursion} goes through directly with these new values.
\end{proof}

As an aside, \cref{fig:farey_multipliers} shows the picture we get when we just draw the edges of the Farey graph corresponding to Farey neighbours which appear as products in some iterate of the recursion:
it is a nice colouring of the Stern-Brocot tree.\footnote{There is a standard bijection in computer science between finite binary strings and vertices of the infinite binary tree; the Stern-Brocot tree
is the tree structure on $ I = (0,1] \inter \Q $ obtained by writing every $ x \in I $ as the image of $ 1/1 $ under a sequence of $ R$'s and $L$'s, where $ R $ and $ L $ are the generators of $ \Gamma_1 $,
the orientation-preserving part of $ \PSL(2,\Z) $ defined at the beginning of \cref{sec:farey}. For a more explicit discussion, see \autocite[Section~6.7]{knuthCM}.}

\begin{figure}
  \centering
  \includegraphics[width=\textwidth]{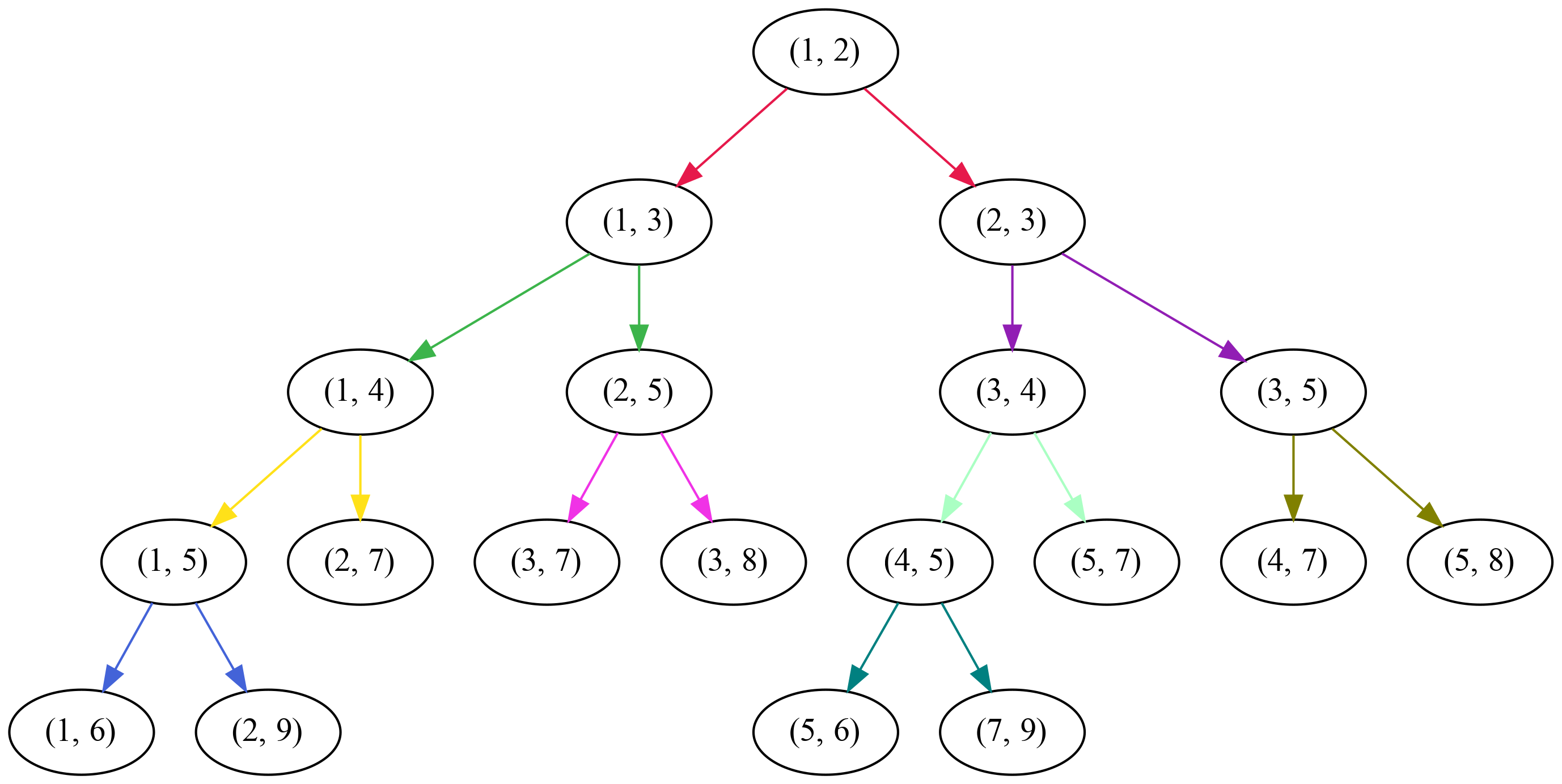}
  \caption{The induced colouring of the Stern-Brocot tree.\label{fig:farey_multipliers}}
\end{figure}

Observe that $ 0/1 $ and $ 1/0 $ are Farey neighbours in $ \hat{\Q} $. Thus, applying the recursion rule \emph{formally} to the diamond
\begin{center}
  \begin{tikzpicture}
    \node[shape=circle, fill=black, minimum size=5pt, inner sep=0pt, outer sep=0pt, label={above:$1/0$}] (A) at (0,1/2) {};
    \node[shape=circle, fill=black, minimum size=5pt, inner sep=0pt, outer sep=0pt, label={below:$1/2$}] (D) at (0,-1/2) {};
    \node[shape=circle, fill=black, minimum size=5pt, inner sep=0pt, outer sep=0pt, label={left:$0/1$}] (B) at (-1/2,0) {};
    \node[shape=circle, fill=black, minimum size=5pt, inner sep=0pt, outer sep=0pt, label={right:$1/1$}] (C) at (1/2,0) {};

    \draw[-stealth'] (A)--(B);
    \draw[-stealth'] (A)--(C);
    \draw[-stealth'] (B)--(C);
    \draw[-stealth'] (B)--(D);
    \draw[-stealth'] (C)--(D);
  \end{tikzpicture}
\end{center}
we obtain
\begin{displaymath}
  \Phi_{0/1} \Phi_{1/1} + \Phi_{1/0} + \Phi_{1/2} = 4 + \frac{1}{\alpha^2} + \alpha^2 + \frac{1}{\beta^2} + \beta^2;
\end{displaymath}
substituting for $ \Phi_{1/1} $, $ \Phi_{1/2} $, and $ \Phi_{0/1} $ from \cref{tab:fareys} we get the following formal expression for $ \Phi_{1/0} $:
\begin{align*}
  \Phi_{1/0} &= 4 + \frac{1}{\alpha^2} + \alpha^2 + \frac{1}{\beta^2} + \beta^2 - \left(\frac\alpha\beta + \frac\beta\alpha - z\right) \left(\alpha\beta + \frac{1}{\alpha\beta} +z\right) - 2&\\
                & {\phantom{{}=4 + \frac{1}{\alpha^2} + \alpha^2 + \frac{1}{\beta^2}+ \beta^2}} - \left(\alpha\beta -\frac\alpha\beta - \frac\beta\alpha + \frac{1}{\alpha\beta}\right)z - z^2\\
    &= 2.
\end{align*}
Observe that $ \Phi_{1/0}^{-1}\left( (-\infty,-2] \right) = \emptyset $, so this is compatible with the Keen--Series theory; it is also a polynomial of degree $ q $ (here, $ q = 0 $)
with constant term $ 2 $, which all agrees with the properties of the higher-degree polynomials. On the other hand, it is not monic!

\section{Commutators of Farey words}\label{sec:commutators}
Recall \textbf{Fricke's identity} for $ A, B \in \PSL(2,\C) $ \autocite[(3.15)]{maclauchlanreid}:
\begin{displaymath}
  \tr [A,B] = \tr^2 A + \tr^2 B + \tr^2 AB - \tr A \tr B \tr AB - 2.
\end{displaymath}
As a direct consequence of \cref{lem:traceid}, we obtain the following special case for Farey words:
\begin{lem}
  If $p/q $ and $ r/s $ are Farey neighbours with $ p/q < r/s $, then
  \begin{multline*}
    \tr [W_{p/q},W_{r/s}] = \Phi_{p/q}^2 + \Phi_{r/s}^2 +\Phi_{(p/q)\oplus(r/s)}^2 + \Phi_{p/q}\Phi_{r/s}\Phi_{(p/q)\oplus(r/s)}\\
    + \kappa(p/q,r/s)^2 -\kappa(p/q,r/s)\left(\Phi_{p/q}\Phi_{r/s} + 2\Phi_{(p/q)\oplus(r/s)}\right) -2;
  \end{multline*}
  where
  \begin{displaymath}
    \kappa(p/q,r/s) \coloneq   \begin{cases}
      2 + \alpha^2 + \frac{1}{\alpha^2} & \text{if $ q+s $ is even,} \\
      \alpha\beta + \frac{\alpha}{\beta} + \frac{\beta}{\alpha} + \frac{1}{\alpha\beta} & \text{if $ q+s $ is odd.}
    \end{cases}
  \end{displaymath}

  In particular, if $ \alpha = \beta = 1 $, then
  \begin{multline}\label{eq:complicated_positive}
    \tr [W_{p/q},W_{r/s}] = \Phi_{p/q}^2 + \Phi_{r/s}^2 +\Phi_{(p/q)\oplus(r/s)}^2 + \Phi_{p/q}\Phi_{r/s}\Phi_{(p/q)\oplus(r/s)}\\ -4\left(\Phi_{p/q}\Phi_{r/s} + 2\Phi_{(p/q)\oplus(r/s)}\right) +14.
  \end{multline}
\end{lem}

In the case of the Maskit slice, the commutator of the equivalent of neighbouring Farey words (`slope words') is always $ -2 $, since in that case the commutator represents a closed curve
about the puncture on the torus. Combining the Fricke identity with the equivalent to \cref{lem:traceid} in this case gives that Farey triplets $ (f,g,h)$ of the trace polynomials of the slope words on the punctured torus
satisfy the \textbf{Markoff equation}
\begin{displaymath}
  f^2 + g^2 + h^2 = 3fgh.
\end{displaymath}
A little experimentation shows that no nice identity holds for the commutators of Farey words in a Riley group---nor should we expect any, since the product of two Farey words is usually
strictly loxodromic for any fixed $ \rho $. There are Markoff-like theories for the 4-times punctured sphere, and one is the theory of Beardon which depends on picking hyperbolic representatives
for the simple closed curves \autocite{beardon86}, but the more relevant to us is the theory of Maloni, Palesi, and Tan \autocite{mpt15} which (as we briefly mentioned above) recovers our
recursion: in \S 2.8 \textit{op. cit.}, take $ \boldsymbol{\mu} = (p,q,r,s) = (8,8,8,-28) $ to recover the recursion (their Equations (5)--(7)) and so obtain the $\boldsymbol{\mu}$-Markoff equation
\begin{displaymath}
  f^2 + g^2 + h^2 + fgh = 8f + 8g + 8h - 28.
\end{displaymath}
In other words we apply their theory to representations of the four-holed sphere group on the four generators $ \alpha,\beta,\gamma,\delta $ into $ \SL(2,\C) $ such that the first two generators go
to a parabolic $ X $ and its inverse and the second go to a parabolic $ Y $ and its inverse ($\tr X = \tr Y = 2 $); this is defined geometrically by taking the topological loops corresponding to the
generators of the four-holed sphere group and `pushing' them into the interior of the 3-manifold $ \H^3/\Gamma_\rho $.

In any case we can generalise in a combinatorial manner the process used to compute the commutator in the Maskit case to obtain a second expression for the commutator (and bypass the Markoff theory)
in the Riley slice case by using both parts of \cref{lem:product_word} together. By that lemma,
\begin{displaymath}
  [W_{p/q},W_{r/s}] = W_{p/q} W_{r/s} (W_{r/s} W_{p/q})^{-1} = \widehat{W}_{(p/q) \oplus (r/s)} \cdot \tilde{W}_{(p/q) \oplus (r/s)}^{-1}
\end{displaymath}
where the hat here indicates that the $ (q+s)$th exponent is flipped in sign and the tilde indicates a flip in the $ (2s)$th exponent. We therefore have two cases: either $ q + s < 2s $,
or $ 2s < q + s $. Without loss of generality, assume that $ q + s < 2s $ and write $ W_{p/q \oplus r/s} = u_1 t u_2 t' u_3 $ where $ u_1 $ is the word made up of the first $ (q+s-1) $ letters,
$ t $ is the $ (q+s)$th letter, $ u_2 $ is the $ (q+s+1)$th to $ (2s-1)$th letters, $ t' $ is the $ (2s)$th letter, and $ u_3 $ is the remainder
of the word; then
\begin{multline*}
  \tr [W_{p/q}, W_{r/s}] = \tr u_1 T u_2 t' u_3 (u_1 t u_2 T' u_3)^{-1}\\= \tr u_1 T u_2 t' u_3 U_3 t' U_2 T U_1 = \tr u_1 T u_2 (t')^2 U_2 T U_1.
\end{multline*}
By rotational symmetry, this last trace is equal to $ \tr u_2 (t')^2 U_2 T^2 $. By assumption, $ t $ and $ t' $ are parabolic so there exists some transformation $ M \in \PSL(2,\C) $ (perhaps the identity) such
that $ mt'M = t $. Writing $ u = u_2 m $, the trace becomes $ \tr u t^2 U T^2 = \tr[u,t^2] $. We now use the Fricke identity again, to see that
\begin{align*}
  \tr [u,t^2] &= \tr^2 u + \tr^2 t^2 + \tr^2 ut^2 - \tr u \tr t^2 \tr ut^2 - 2\\
              &= \tr^2 u + \tr^2 ut^2 - 2\tr u \tr ut^2 + 2\\
              &= (\tr u - \tr ut^2)^2 + 2.
\end{align*}

This proves the following result, which may be known to experts:
\begin{lem}
  If $ p/q $ and $ r/s $ are Farey neighbours, then $ \tr [W_{p/q}(z), W_{r/s}(z)] -2 $ is a square in $ \Z[z] $. \qed
\end{lem}
As an amusing corollary, we see that the complicated expression to the right of \cref{eq:complicated_positive} is a positive function.

\section{Closed-form solutions for some cases}\label{sec:properties}
Recall that the \textbf{Chebyshev polynomials} (of the first kind) are the family of polynomials $ T_n $ defined via the recurrence relation
\begin{gather*}
  T_0(x) = 1\\
  T_1(x) = x\\
  T_{n+1}(x) = 2x T_n(x) - T_{n-1}(x).
\end{gather*}
It is well-known that these polynomials satisfy the product relation
\begin{displaymath}
  2T_m(x) T_n(x) = T_{m+n}(x) + T_{\abs{m-n}}(x)
\end{displaymath}
for $ m,n \in \Z_{\geq 0} $. Compare this relation with the parabolic recurrence relation of \cref{cor:parabolic_identity} (but note that the Chebyshev product rule
holds for all $ m,n $ and the identities for the Farey polynomials hold only for Farey neighbours).

We may apply the theory of `Farey recursive functions' \autocite{chesebro19,chesebro20} in order to explain this analogy.
The following diagram may be useful for translating the notation of that paper (right: $ \gamma \oplus^2 \alpha $ is their
notation for $ \gamma \oplus \alpha \oplus \alpha $) into the notation we use here (left):
\begin{center}
  \begin{tikzpicture}
    \node[shape=circle, fill=black, minimum size=5pt, inner sep=0pt, outer sep=0pt, label={above:$\beta \ominus \alpha$}] (A) at (-1,1/2) {};
    \node[shape=circle, fill=black, minimum size=5pt, inner sep=0pt, outer sep=0pt, label={below:$\beta \oplus \alpha$}] (D) at (-1,-1/2) {};
    \node[shape=circle, fill=black, minimum size=5pt, inner sep=0pt, outer sep=0pt, label={left:$\alpha$}] (B) at (-3/2,0) {};
    \node[shape=circle, fill=black, minimum size=5pt, inner sep=0pt, outer sep=0pt, label={right:$\beta$}] (C) at (-1/2,0) {};

    \draw[-stealth'] (A)--(B);
    \draw[-stealth'] (A)--(C);
    \draw[-stealth'] (B)--(C);
    \draw[-stealth'] (B)--(D);
    \draw[-stealth'] (C)--(D);

    \node[shape=circle, fill=black, minimum size=5pt, inner sep=0pt, outer sep=0pt, label={above:$\gamma$}] (A') at (1,1/2) {};
    \node[shape=circle, fill=black, minimum size=5pt, inner sep=0pt, outer sep=0pt, label={below:$\gamma \oplus^2 \alpha$}] (D') at (1,-1/2) {};
    \node[shape=circle, fill=black, minimum size=5pt, inner sep=0pt, outer sep=0pt, label={left:$\alpha$}] (B') at (1/2,0) {};
    \node[shape=circle, fill=black, minimum size=5pt, inner sep=0pt, outer sep=0pt, label={right:$\gamma \oplus \alpha$}] (C') at (3/2,0) {};

    \draw[-stealth'] (A')--(B');
    \draw[-stealth'] (A')--(C');
    \draw[-stealth'] (B')--(C');
    \draw[-stealth'] (B')--(D');
    \draw[-stealth'] (C')--(D');
  \end{tikzpicture}
\end{center}

Chesebro et. al. \autocite[Definition 3.1]{chesebro20} define the notion of a Farey-recursive function: this is a function on $ \Q $ which satisfies a recurrence on the Farey diagram of a
particular form. In order to incorporate our polynomials into their framework, we need to make a slight generalisation in order to allow for constant terms in the recurrence
relation.

\begin{defn}
  Let $ R $ be a (commutative) ring, and suppose $ d_1,d_2,d_3 : \hat{\Q} \to R $. A function $ \mc{F} : \hat{\Q} \to R $ is a \textbf{$(d_1,d_2)$-Farey recursive function} if, whenever
  $ \alpha,\beta \in \hat{\Q} $ are Farey neighbours,
  \begin{displaymath}
    \mc{F}(\beta\oplus \alpha) = -d_1(\alpha) \mc{F}(\beta\ominus\alpha) + d_2(\alpha) \mc{F}(\beta) + d_3(\alpha).
  \end{displaymath}
  If $ d_3 $ is the zero function, we say that the recurrence relation is homogeneous.
\end{defn}

\begin{ex}\label{ex:farey_are_recursive}
  The parabolic Farey polynomials are Farey recursive functions. To make this clearer, we rewrite our recurrence relation slightly as
  \begin{displaymath}
    \Phi(\beta \oplus \alpha) = 8 - \Phi(\beta \ominus \alpha) - \Phi(\alpha) \Phi(\beta)
  \end{displaymath}
  (where we set $ \beta = p/q $, $ \alpha = r/s $, and changed from subscript notation to functional notation). In our case, then,
  $ d_1(\alpha) $ is constantly 1 and $ d_2 = -\Phi $.
\end{ex}

One easily obtains the relevant generalisations of the existence-uniqueness results of \autocite[Section 4]{chesebro20} (the same proofs work, with the usual
property that the space of non-homogeneous solutions is the sum of a particular solution and the space of homogeneous solutions).

There is an obvious explicit solution to the non-homogeneous Farey polynomial recursion in \cref{ex:farey_are_recursive}: namely, the
map $ \Phi $ which sends every $ \alpha \in \Q $ to the constant polynomial $ 2 \in \Z[z] $. It therefore remains to solve the
corresponding homogeneous equation in any case of interest.

\subsection{A Fibonacci-like sequence of homogeneous Farey polynomials}\label{sec:diagonalise}

\begin{notation}
  In this section, we work exclusively with the parabolic Farey polynomials, $ \Phi^{\infty,\infty}_{p/q} $. We will reuse the symbols
  $ \alpha $ and $ \beta $ for rational numbers (since the roots of unity used to define $ X $ and $ Y $ become 1 in this situation and
  so there is no need for special notation for them).
\end{notation}

In \cref{tab:farey_homog}, we list the first few polynomials $ \Phi^h $ which solve the homogeneous recursion relation
\begin{equation}\label{eqn:identity_frf_homogeneous}
  \Phi^h(\beta \oplus \alpha) = - \Phi^h(\beta \ominus \alpha) - \Phi^h(\alpha) \Phi^h(\beta)
\end{equation}
with the initial values $ \Phi^h(0/1) = 2-z $, $ \Phi^h(1/0) = 2 $, and $ \Phi^h(1/1) = 2+z $.

\begin{table}
  \centering
  \caption{Selected polynomials $ \Phi^h $ satisfying the homogeneous recursion relation for small $ q $, with the initial values as given.\label{tab:farey_homog}}
  \begin{tabular}{ccp{0.8\textwidth}}
    \toprule
    $p$ & $q$ & $ \Phi^h_{p/q} $ \\\midrule
    1&0&$2$\\
    0&1&$-z+2$\\
    1&1&$z+2$\\
    1&2&$z^2-6$\\
    1&3&$z^3-2z^2-7z+10$\\
    2&3&$-z^3-2z^2+7z+10$\\
    1&4&$z^4 - 4 z^3 - 4 z^2 + 24 z - 14$\\
    3&4&$z^4 + 4 z^3 - 4 z^2 - 24 z - 14$\\
    1&5&$z^5 - 6 z^4 + 3 z^3 + 34 z^2 - 55 z + 18$\\
    2&5&$-z^5 + 2 z^4 + 13 z^3 - 22 z^2 - 41 z + 58$\\
    3&5&$z^5 + 2 z^4 - 13 z^3 - 22 z^2 + 41 z + 58$\\
    4&5&$-z^5 - 6 z^4 - 3 z^3 + 34 z^2 + 55 z + 18$\\
    1&6&$z^6-8 z^5+14 z^4+32 z^3-119 z^2+104 z-22$\\
    1&7&$z^7-10 z^6+29 z^5+10 z^4-186 z^3+308 z^2-175 z+26$\\
    1&8&$z^8-12 z^7+48 z^6-40 z^5-220 z^4+648 z^3-672 z^2+272 z-30$\\
    1&9&$z^9 - 14 z^8 + 71 z^7 - 126 z^6 - 169 z^5 + 1078 z^4 - 1782 z^3 +  1308 z^2 - 399 z + 34$\\
    1&10&$z^{10} - 16 z^9 + 98 z^8 - 256 z^7 + 35 z^6 +  1456 z^5 - 3718 z^4 + 4224 z^3 - 2343 z^2 + 560 z - 38$\\\bottomrule
 \end{tabular}
\end{table}
The polynomials with numerator 1 listed in the table have very nice properties: immediately one sees that the constant terms alternate in sign and increase in magnitude by 4 each time;
also, we have that $ \Phi^h_{1/q}(1) $ cycles through the values $ 3, -5, 2 $, $ \Phi^h_{1/q}(2) $ cycles through the values $ 4, -2, -4, 2 $; and when we evaluate at 3 and 4
we get a 6-cycle and an arithmetic sequence of step 4 respectively. When we consider $ \Phi^h(1/q)(5) $, though, we obtain more interesting behaviour: this is OEIS
sequence A100545\footnote{\url{http://oeis.org/A100545}} and satisfies the Fibonacci-type relation
\begin{displaymath}
  \Phi^h_{1/q}(5) = 3\Phi^h_{1/(q-1)}(5) - \Phi^h_{1/(q-2)}(5) \quad \text{ with } \Phi^h_{1/1}(5)=7, \Phi^h_{1/2}(5)=19.
\end{displaymath}
Of course, from the way that we defined the $ \Phi^h $ such types of relations ought to be expected. In this section, we use the standard diagonalisation technique to explain the behaviour
of the sequence $ a_q \coloneq \Phi^h(1/q)(z) $ for fixed $ z \in \C $. From \cref{eqn:identity_frf_homogeneous}, we have that
\begin{displaymath}
  a_q = -(2-z)a_{q-1} - a_{q-2}.
\end{displaymath}
We may rewrite this equation in matrix form as the following:
\begin{displaymath}
  \begin{bmatrix}
    0 & 1\\
    -1 & z-2
  \end{bmatrix}
  \begin{bmatrix} a_{q-2} \\ a_{q-1} \end{bmatrix} = \begin{bmatrix} a_{q-1}  \\ a_q \end{bmatrix}.
\end{displaymath}
One easily computes that the eigenvalues of the transition matrix are
\begin{displaymath}
  \lambda^{\pm} = \frac{z-2\pm \alpha}{2}
\end{displaymath}
(where $ \alpha = \sqrt{z^2-4z} $) with respective eigenvectors
\begin{displaymath}
  v^\pm = \begin{bmatrix} z-2 \mp \alpha \\ 2 \end{bmatrix}
\end{displaymath}
(note the alternated sign). Thus the transition matrix may be diagonalised as
\begin{equation}\label{eqn:diagonalised}
  \frac{-1}{2\alpha}
  \begin{bmatrix} z-2 - \alpha & z-2 + \alpha \\ 2 & 2 \end{bmatrix}
  \begin{bmatrix} \frac{z-2 + \alpha}{2} & 0 \\ 0 & \frac{z-2 - \alpha}{2} \end{bmatrix}
  \begin{bmatrix} 2 & 2-z - \alpha \\ -2 & z-2 - \alpha \end{bmatrix}
\end{equation}
and so $ a_q $ is the first coordinate of
\begin{displaymath}
  \frac{-1}{2\alpha}
  \begin{bmatrix} z-2 - \alpha & z-2 + \alpha \\ 2 & 2 \end{bmatrix}
  \begin{bmatrix} \left(\frac{z-2 + \alpha}{2}\right)^q & 0 \\ 0 & \left(\frac{z-2 - \alpha}{2}\right)^q \end{bmatrix}
  \begin{bmatrix} 2 & 2-z - \alpha \\ -2 & z-2 - \alpha \end{bmatrix}
  \begin{bmatrix}
    a_0\\
    a_1
  \end{bmatrix};
\end{displaymath}
expanding out, we get the general solution
\begin{displaymath}
  \begin{aligned}
  a_q &= 2^{-1-q} \alpha^{-1} \bigg( \left(a_0\left(z-2+\alpha\right) -2a_1\right) \left(z-2 - \alpha\right)^q\\
      &\hspace{6em}{}+ \left(a_0\left(2-z+\alpha\right) + 2a_1\right) \left(z-2 + \alpha\right)^q\bigg).
  \end{aligned}
\end{displaymath}

We may also characterise the $ z $ for which $ \Phi^h_{1/q}(z) $ is cyclic: this occurs precisely when the diagonal matrix
of \cref{eqn:diagonalised} is of finite order, i.e. whenever both $ (z-2 \pm \alpha)/2 $ are roots of unity.

\begin{ex}
  As an application of the theory above, we have seen that the Chebyshev polynomials also satisfy a second-order recurrence relation with transition matrix
  \begin{displaymath}
    \begin{bmatrix} 0 & 1 \\ -1 & 2x \end{bmatrix}.
  \end{displaymath}
  If we set $ z = 2x + 2 $, then we get back the transition matrix we derived for our homogeneous recurrence. Thus our sequence $ \Phi^h_{1/q}(z) $ for
  fixed $ z $ is of the form $ W_q(\frac{z-2}{2}) $ where $ W_q $ is the $q$th Chebyshev polynomial in the sequence beginning with $ W_0 = 2x $ and $ W_1 = 2x + 4 $.
\end{ex}

Finally, we consider the solution of the non-homogeneous equation for $ \Phi_{1/q} $. Above, we observed that there is a constant solution to the global recursion
relation on the entire Stern-Brocot tree; we therefore guess that there is a similar solution to this recursion. Such a solution $ f $ will satisfy
\begin{displaymath}
  8 = f(z) + (2-z)f(z) + f(z)
\end{displaymath}
and arithmetic gives $ f(z) = 8/(4-z) $. Combining this with the general solution above gives us the following general solution to the non-homogeneous relation:
\begin{displaymath}
  \begin{aligned}
  a_q &= \frac{8}{4-z} + 2^{-1-q}\alpha^{-1} \bigg((\lambda(z-2+\alpha) -2\mu) (z-2 - \alpha)^q\\
  &\hspace{10em}{}+ (\lambda(2-z+\alpha) + 2\mu) (z-2 + \alpha)^q\bigg).
  \end{aligned}
\end{displaymath}
In our case, we have $ a_0 = \Phi_{1/0}(z) = 2 $ and $ a_1 = \Phi_{1/1}(z) = 2+z $. Solving the resulting system of equations gives
\begin{displaymath}
  (\lambda,\mu) = \left(\frac{2z}{z-4}, \frac{2z-z^2}{z-4}\right)
\end{displaymath}
and hence
\begin{displaymath}
  a_q = \frac{8}{4 - z} + \frac{2^{-q}z}{z-4} \left( (-2 + z - \sqrt{z^2 -4 z})^q + (-2 + z + \sqrt{z^2 -4 z})^q \right)
\end{displaymath}

\subsection{Solving the homogeneous recursion relation in general}
In the previous section, we computed a closed form formula for $ \Phi^h_{1/q}(n) $ using standard techniques from the theory of second-order
linear recurrences. We now tackle the general problem of finding a closed-form formula for $ \Phi^h_{p/q}(n) $; in order to do this, we use
the theory of Section 6 of \autocite{chesebro20} but with a slight modification: in that paper, the authors define a special case of Farey recursive function,
a Farey recursive function of determinant $ d $ (where $ d : \hat{\Q} \to R $), to be a Farey recursive function $ \mc{F} $ with $ d_1 = d $ and $ d_2 = \mc{F} $.
Our situation is very similar, except that instead of $ d_2 = \mc{F} $ we have $ d_2 = -\mc{F} $. To reflect this, we will call a Farey recursive function satisfying
a relation of the form
\begin{equation}\label{eqn:antidenom}
  \mc{F}(\beta\oplus \alpha) = -d(\alpha) \mc{F}(\beta\ominus\alpha) - \mc{F}(\alpha) \mc{F}(\beta).
\end{equation}
a Farey recursive function of \emph{anti}-determinant $ d $. We will work for the time being in this
setting (i.e. we will work with the general function $ \mc{F} $ rather than the particular
example $ \Phi $) in order to restate in sufficient generality the theorem which we need (Theorem 6.1 of \autocite{chesebro20}).

Let $ \alpha \in \Q $. The boundary sequence $ \partial(\alpha) $ is defined inductively by the process of `continuing to expand down the Farey graph by constant steps'.
More precisely, let $ \beta \oplus^k \gamma $ denote $ ((\beta \underbrace{\oplus \gamma) \oplus \cdots) \oplus \gamma}_{k\text{ iterates}} $ for $ \beta,\gamma \in \hat{\Q} $
and let $ \gamma_L,\gamma_R $ be the unique Farey neighbours such that $ \alpha = \gamma_L \oplus \gamma_R $; then we set
\begin{displaymath}
  \partial(\alpha) \coloneq \{ \gamma_L \oplus^k \alpha : k \in \Z_{\geq 0} \} \cup \{ \gamma_R \oplus^k \alpha : k \in \Z_{\geq 0} \}.
\end{displaymath}
If we allow the Farey graph to embed in the Euclidean upper halfplane by sending $ \Q \ni p/q \mapsto (p/q,1/q) \in \H^2 $ (note: this is not the Farey triangulation), then except for the exceptional cases $ \alpha = 1/0 $
and $ \alpha = n/1 $ for $ n \in \Z $ the subgraph spanned by $ \partial(\alpha) $ corresponds to a Euclidean triangle containing $ \alpha $ in its interior, see
Figure 3 of \autocite{chesebro20}; for example, the triangle spanned by $ \partial(1/n) $ is the triangle with vertices $ 0, (1/2,1/2), 1 $.

It will be useful to have specific names for the terms in each of the two subsequences and so we set, for $ k \in \Z $,
\begin{equation}\label{eqn:betas}
  \beta_{k} \coloneq
  \begin{cases}
    \gamma_L \oplus^{-k-1} \alpha & \text{if $ k < -1 $}\\
    \gamma_L & \text{if $ k = -1 $}\\
    \gamma_R & \text{if $ k = 0 $}\\
    \gamma_R \oplus^k \alpha & \text{if $ k > 0 $}.
  \end{cases}
\end{equation}

For every $ \alpha \in \Q $, define
\begin{equation}\label{eqn:bdry_matrix}
  M_\alpha = \begin{bmatrix} 0 & 1 \\ -d(\alpha) & \mc{F}(\alpha) \end{bmatrix}.
\end{equation}
Given any Farey neighbour $ \beta $ of $ \alpha $, we have
\begin{displaymath}
  M_\alpha^n \begin{bmatrix} \mc{F}(\gamma \oplus^0 \alpha)\\\mc{F}(\gamma\oplus^1\alpha) \end{bmatrix} = \begin{bmatrix} \mc{F}(\gamma \oplus^n \alpha)\\\mc{F}(\gamma\oplus^{n+1}\alpha) \end{bmatrix}
\end{displaymath}
and so the recursion \cref{eqn:antidenom} is equivalent to a family of second-order linear recurrences, one down $ \partial(\alpha) $ for each $ \alpha $.

We may now state the following theorem:
\begin{thm}[Adaptation of Theorem 6.1 of \autocite{chesebro20}]\label{thm:round_the_triangles}
  Let $ d : \hat{\Q} \to R $ be a multiplicative function (in the sense that $ d(\gamma\oplus\beta)=d(\gamma)d(\beta)$ for all pairs of Farey neighbours $\beta,\gamma\in\Q $) to a
  commutative ring $ R $, such that $ d(\hat{\Q}) $ contains no zero divisors. Suppose that $ \mc{F} $ is a Farey recursive function with anti-determinant $ d $. Given $ \alpha \in \Q $,
  define $ M_\alpha $ as in \cref{eqn:bdry_matrix} and $ (\beta_k)_{k\in\Z} $ as in \cref{eqn:betas}.

  Then, for all $ n \in \Z $,
  \begin{displaymath}
    M_\alpha^n  \begin{bmatrix} \mc{F}(\beta_0) \\ \mc{F}(\beta_1) \end{bmatrix} =
    \begin{cases}
      \begin{bmatrix} \mc{F}(\beta_n)\\\mc{F}(\beta_{n+1})\end{bmatrix} & n \geq 0,\\[1.5em]
      \begin{bmatrix} \frac{1}{d(\beta_{-1})} \mc{F}(\beta_{-1}) \\ \mc{F}(\beta_0) \end{bmatrix} & n = -1,\text{ and}\\[1.5em]
      \begin{bmatrix} \frac{1}{d(\beta_{-1}) d_{\alpha}^{-n-1}} \mc{F_{\beta_n}} \\ \frac{1}{d(\beta_{-1}) d_{\alpha}^{-n-2}} \mc{F_{\beta_{n+1}}} \end{bmatrix} & n < -1.
    \end{cases}
  \end{displaymath}
\end{thm}

We proceed to prove \cref{thm:round_the_triangles} by exactly the same argument as given in \autocite{chesebro20}. The key point is the following lemma, which is the analogue of the discussion
directly preceeding the statement of Theorem 6.1 in that paper.
\begin{lem}\label{lem:round_the_triangles_prelim}
  With the setup of \cref{thm:round_the_triangles}, we have
  \begin{gather*}
    M_\alpha^{-1} \begin{bmatrix} \mc{F}(\beta_0) \\ \mc{F}(\beta_1) \end{bmatrix} = \begin{bmatrix} \frac{d(\beta_0)}{d(\alpha)} \mc{F}(\beta_{-1}) \\ \mc{F}(\beta_0) \end{bmatrix}\\
    M_\alpha^{-2} \begin{bmatrix} \mc{F}(\beta_0) \\ \mc{F}(\beta_1) \end{bmatrix} = \begin{bmatrix} \frac{1}{d(\alpha) d(\beta_{-1})} \mc{F}(\beta_{-2}) \\ \frac{1}{d(\beta_{-1})} \mc{F}(\beta_{-1}) \end{bmatrix}
  \end{gather*}
\end{lem}
\begin{proof}
  The formula involving $ M_\alpha^{-1} $ comes directly from computing the product on the left via the definition and simplifying with the formula
  \begin{displaymath}
    \mc{F}(\beta_1) = -d(\beta_0)\mc{F}(\beta_{-1}) - \mc{F}(\alpha)\mc{F}(\beta_0)
  \end{displaymath}
  which is almost exactly the same as Equation (8) of \autocite{chesebro20}---the single sign change cancels exactly with the sign change
  between the `determinant' and `anti-determinant' recurrences so we get the same overall formula for the $ M_\alpha^{-1} $ product as they do in Equation (11) of their paper.

  The formula for $ M_\alpha^{-2} $ comes from applying the analogues of Equations (9) and (10) of their paper,
  \begin{gather*}
    \mc{F}(\beta_{-2}) = -d(\beta_{-1}) \mc{F}(\beta_0) - \mc{F}(\alpha)\mc{F}(\beta_{-1})\\
    d(\alpha) = d(\beta_{-1}) d(\beta_0)
  \end{gather*}
  and simplifying; again the minus signs cancel and we get the same formula.
\end{proof}

\begin{proof}[Proof of \cref{thm:round_the_triangles}]
  The formula for $ n \geq 0 $ holds for all Farey recursive formulae as noted above; the formulae for $ n = -1 $ and $ n = -2 $ are just
  the formulae of \cref{lem:round_the_triangles_prelim}; and we proceed to prove the formula for $ n < -2 $ by induction. Assume that the formula
  holds for some fixed $ n \leq -2 $; then from the definitions we have
  \begin{displaymath}
    \mc{F}(\beta_{n-1}) = -\mc{F}(\alpha)\mc{F}(\beta_n) - d(\alpha)\mc{F}(\beta_{n+1})
  \end{displaymath}
  and so we can compute
  \begin{align*}
    M_\alpha^{n-1} \begin{bmatrix} F(\beta_0)\\F(\beta_1) \end{bmatrix}
      &= M_{\alpha}^{-1} M_\alpha^n \begin{bmatrix} F(\beta_0)\\F(\beta_1) \end{bmatrix}\\
      &= \frac{1}{d(\alpha)} \begin{bmatrix} -\mc{F}(\alpha) & -1 \\ d(\alpha) & -0 \end{bmatrix}
          \begin{bmatrix} \frac{1}{d(\beta_{-1}) d(\alpha)^{-n-1}}F(\beta_n)\\ \frac{1}{d(\beta_{-1}) d(\alpha)^{-n-2}}F(\beta_{n+1}) \end{bmatrix}\\
      &= \frac{1}{d(\alpha)} \begin{bmatrix} -\frac{1}{d(\beta_{-1}) d(\alpha)^{-n-1}} \left( \mc{F}(\alpha)\mc{F}(\beta_n) + d(\alpha) \mc{F}(\beta_{n+1}) \right)\\
                                             \frac{1}{d(\beta_{-1}) d(\alpha)^{-n-2}} \mc{F}(\beta_{n}) \end{bmatrix}\\
      &= \begin{bmatrix} -\frac{1}{d(\beta_{-1}) d(\alpha)^{-n}} \mc{F}(\beta_{n-1})\\
                                             \frac{1}{d(\beta_{-1}) d(\alpha)^{-n-1}} \mc{F}(\beta_{n}) \end{bmatrix}
  \end{align*}
  which is the desired result.
\end{proof}

\begin{cor}[Adaptation of Corollary 6.2 of \autocite{chesebro20}]\label{cor:determinant1}
  Let $ \Phi^h $ be a family of homogeneous Farey polynomials (i.e. a family solving \cref{eqn:identity_frf_homogeneous} for some starting values). Then,
  for some $ \alpha \in \Z $, if $ M_\alpha $ is the matrix
  \begin{displaymath}
    \begin{bmatrix}0&1\\-1&\Phi^h(\alpha)\end{bmatrix}
  \end{displaymath}
  and if $ (\beta_k)_{k\in\Z} $ are the boundary values about $ \alpha $ as in \cref{eqn:betas}, then for all $ n \in \Z $ we have
  \begin{displaymath}
    M^n_\alpha \begin{bmatrix} \Phi^h(\beta_0) \\ \Phi^h(\beta_1) \end{bmatrix} = \begin{bmatrix} \Phi^h(\beta_n) \\ \Phi^h(\beta_{n+1}) \end{bmatrix}.
  \end{displaymath}
\end{cor}
\begin{proof}
  This follows directly from \cref{thm:round_the_triangles} with the observation that the anti-determinant of $ \Phi^h $ is the constant function $ d(\gamma) = 1 $ for all $ \gamma \in \Q $.
\end{proof}

Thus to determine $ \Phi^h_{\alpha} $ for all $ \alpha \in \Q $ it suffices to compute and diagonalise the $ M_\alpha $ matrices, using the techniques
of \cref{sec:diagonalise}. (Of course, we need to diagonalise in the ring of rational functions over $ \Q $ rather than the ring of polynomials over $ \Z $.)
More precisely, we need to compute $ M_{\alpha_i} $ for some family $ (\alpha_i) $ of rationals with the property that the boundary sets $ \partial(\alpha_i) $
cover $ \Q $. (In \cref{sec:diagonalise}, we did this computation for $ \partial(0/1) $.)

In any case, from \cref{cor:determinant1} we immediately have a qualitative result:
\begin{thm}\label{thm:qualitative_chebyshev}
  For any $ \gamma \in \Q $, there exists a sequence $ \ldots,\gamma_{-1},\gamma_0 = \gamma, \gamma_1, \gamma_2,\ldots $ of rational numbers such that $ \Phi^h_{\gamma_n}(z) $
  is a sequence of Chebyshev polynomials $ W_n(\Phi^h_\gamma(z)/2) $.  (Namely, let $ \gamma_{-1} $ be a neighbour in the Stern-Brocot tree of $ \gamma $ and take the sequence $ (\gamma_k) $
  to be precisely the sequence $ (\beta_k) $ of \cref{eqn:betas} with $ \alpha \coloneq \gamma \ominus \gamma_{-1} $.) \qed
\end{thm}
\begin{rem}
  Of course, the boundary sequence $ (\gamma_k) $ constructed here is just a geodesic line $ \Lambda $ in the Stern-Brocot tree rooted at $ \gamma $, defined by choosing
  one vertical half-ray in the tree starting from $ \gamma $ (where `vertical' refers to the embedding of \cref{fig:farey_multipliers}) and then extending that in the Farey graph in the obvious
  way by repeated Farey arithmetic with the same difference. There are clearly two such natural choices for $ \Lambda $ given a fixed $ \gamma $ ($ \gamma $ has three neighbours,
  but two correspond to the same geodesic), and a single natural choice is obtained by taking the unique neighbour of $ \gamma $ which lies above.
\end{rem}

We easily compute that the eigenvalues of $ M_\alpha $ are
\begin{displaymath}
  \lambda^{\pm} = \frac{1}{2}\left(\Phi^h_\alpha \pm \sqrt{(\Phi^h_\alpha)^2 - 4} \right)
\end{displaymath}
Let $ x = \Phi^h_\alpha $ and $ \kappa = \sqrt{x^2 - 4} $ (this is the analogue of the constant $ \alpha $ from \cref{sec:diagonalise}); then the respective eigenvectors are
\begin{displaymath}
  v^{\pm} = \begin{bmatrix} x \mp \kappa \\ 2 \end{bmatrix}.
\end{displaymath}
We therefore may diagonalise $ M_\alpha $ as
\begin{displaymath}
  M_\alpha = -\frac{1}{4\kappa} \begin{bmatrix} x-\kappa & x+\kappa \\ 2 & 2 \end{bmatrix}
                                \begin{bmatrix} \frac{1}{2}(x+\kappa) & 0 \\ 0 & \frac{1}{2}(x-\kappa) \end{bmatrix}
                                \begin{bmatrix} 2 & -x-\kappa \\ -2 & x-\kappa \end{bmatrix};
\end{displaymath}
in particular, $ \Phi^h(\beta_n) $ is the first component of
\begin{align*}
  &M_\alpha^n \begin{bmatrix} \Phi^h(\beta_0) \\ \Phi^h(\beta_1) \end{bmatrix}\\
    &\quad{}= -\frac{1}{4\kappa} \begin{bmatrix} x-\kappa & x+\kappa \\ 2 & 2 \end{bmatrix}
                         \begin{bmatrix} \frac{1}{2^n}(x+\kappa)^n & 0 \\ 0 & \frac{1}{2^n}(x-\kappa)^n \end{bmatrix}
                         \begin{bmatrix} 2 & -x-\kappa \\ -2 & x-\kappa \end{bmatrix}
                         \begin{bmatrix} \Phi^h(\beta_0) \\ \Phi^h(\beta_1) \end{bmatrix}
\end{align*}
computing this, we have
\begin{displaymath}
  \Phi^h(\beta_n) = \frac{\splitfrac{(\Phi^h(\beta_0) (x+\kappa) - 2\Phi^h(\beta_1)) (x - \kappa)^n}{ {}+ (\Phi^h(\beta_0) (\kappa-x) +2\Phi^h(\beta_1)) (x + \kappa)^n}}{2^{1+n}\kappa}.
\end{displaymath}

In particular, we have proved the following quantitative improvement of \cref{thm:qualitative_chebyshev}:
\begin{thm}
  Let $ \beta_0 $ and $ \beta_1 $ be Farey neighbours, and let $ \alpha = \beta_1 \ominus \beta_0 $. Then we have a closed form formula for $ \Phi^h(\beta_n) $ ($ n \in \Z $), namely
  \begin{displaymath}
    \Phi^h_{\beta_n} = \frac{\left(\Phi^h_{\beta_0} \left(\Phi^h_{\alpha}+\kappa\right) - 2\Phi^h_{\beta_1}\right) \left(\Phi^h_{\alpha} - \kappa\right)^n + \left(\Phi^h_{\beta_0} \left(\kappa-\Phi^h_{\alpha}\right) +2\Phi^h_{\beta_1}\right) \left(\Phi^h_{\alpha} + \kappa\right)^n}{2^{1+n}\kappa}.
  \end{displaymath}
  where $ \kappa = \sqrt{\left(\Phi^h_\alpha\right)^2 - 4} $. \qed
\end{thm}

This gives a `local' closed form solution for the recursion around any $ \alpha \in \Q $; a `global' solution corresponds to a collection of these solutions, each local to a particular
geodesic in the graph and which are compatible on intersections. Unfortunately, our original recurrence relied on knowing only three initial values globally in the graph; while this local formula relies
on knowing three initial values which are local on the particular geodesic.

\section{Final remarks on the Farey polynomials}\label{sec:contin_frac}
As we mentioned in the introduction to \autocite{ems21}, a version of this theory can be used to give approximations to irrational pleating rays. In order to do this,
we must deal with the theory of infinite continued fractions.

Every rational number can be expressed as a finite simple continued fraction: if $ (a_0,a_1,\ldots,a_k) $ is a finite sequence of integers, we define the simple continued fraction
\begin{displaymath}
  [a_0;a_1,\ldots,a_k] \coloneq a_0 + 1/\left(a_1 + 1/\left( 1/\left(a_2 + 1/\left( \cdots + 1/a_k\right)\right)\right)\right)
\end{displaymath}
and every every rational number $ p/q $ has exactly two expressions as such a continued fraction \autocite[Theorem 162]{hardywright}.
It is also a well-known result in classical number theory that every \emph{irrational} $ \lambda \in \R\setminus\Q $ has a unique simple continued fraction approximation of the form
\begin{displaymath}
  \lambda = [a_0;a_1,\ldots,a_n,\ldots]
\end{displaymath}
which can be computed efficiently by repeated application of the Euclidean algorithm (see, for example, \S 10.9 of \autocite{hardywright}). The following result,
which is obtained by combining Theorem~149 and Theorem~150 of \autocite{hardywright}, exhibits $ \lambda $ as a limit of a sequence of rationals `down the Farey tree'.

\begin{prp}\label{prp:farey_expansion}
  Suppose that $ p/q = [a_1,\ldots,a_{N-1},a_N,1] $; define
  \begin{displaymath}
    \frac{r_1}{s_1} = [a_1,\ldots,a_{N-1},a_N] \text{ and } \frac{r_2}{s_2} = [a_1,\ldots,a_{N-1}].
  \end{displaymath}
  Then $ r_1/s_1 $ and $ r_2/s_2 $ are Farey neighbours and $ p/q = (r_1/s_1) \oplus (r_2/s_2) $. \qed
\end{prp}

In the previous section, we indicated how to compute in closed form the sequence of Farey polynomials corresponding to the Farey fractions
\begin{displaymath}
  \frac{p_1}{q_1}, \frac{p_2}{q_2} = \frac{p_1}{q_1} \oplus \left(\frac{p_2}{q_2} \ominus \frac{p_1}{q_1}\right), \ldots, \frac{p_n}{q_n} = \frac{p_1}{q_1} \oplus^{n-1} \left(\frac{p_2}{q_2} \ominus \frac{p_1}{q_1}\right)
\end{displaymath}
where $ p_1/q_1 $ and $ p_2/q_2 $ are Farey neighbours. That is, we gave a way to compute the Farey polynomials down a branch of the Farey tree with constant difference
(for instance, we gave the example of $ \Phi_{1/q} $, where $\frac{1}{q} = \frac{1}{0} \oplus^q \frac{0}{1} $). The study of partial fraction decompositions here gives, in general,
different sequences: the constant addition sequence rooted at an element $ \xi $ in the tree is the sequence which constantly chooses the \emph{left} branch when moving down
from $ \xi $ (with respect to the embedding of \cref{fig:farey_multipliers}), while the sequence corresponding to continually adding the previous two items in the tree (and therefore
building a continued fraction decomposition) corresponds to the sequence which is eventually constantly moving \emph{rightwards}. Unfortunately we were not able to write down closed-form
expressions for sequences of Farey polynomials corresponding to finite convergents of infinite continued fraction decompositions in any more general way.

\begin{rem}
  The geometric meaning of all this comes from the theory of two-bridge links and is described in \autocite[Chapter 12]{burde}; one is really writing down a sequence of 2-bridge links
  which converge to the link of interest, increasing by one crossing in each step.
\end{rem}

The recursion relation of interest comes, as always, from applying the Farey polynomial operator to the recursion relation $ a_n = a_{n-1} \oplus a_{n-2} $; doing this we obtain
\begin{displaymath}
  \Phi_{a_n} = \Phi_{a_{n-1} \oplus a_{n-2}} = 8 - \Phi_{a_{n-1} \ominus a_{n-2}} - \Phi_{a_{n-1}} \Phi_{a_{n-2}}.
\end{displaymath}

Observe that $ a_{n-1} \ominus a_{n-2} $ is just $ a_{n-3} $ and replace the cumbersome notation $ \Phi_{a_k} $ with $ x_k $ to get the relation
\begin{equation}\label{eq:recrelation}
  x_n = 8 - x_{n-3} - x_{n-2}x_{n-1}.
\end{equation}
Unfortunately, this is a non-linear recurrence relation and we are unable to find a closed-form solution. (We believe that if one exists then it will be
in terms of combinatorial objects, e.g. Stirling numbers.) However, it does at least give a computationally feasible method for approximating irrational cusp points.

\begin{ex}[The Fibonacci polynomials]
  Let us compute an approximation to the pleating ray with asymptotic angle $ \pi/\phi $, where $ \phi $ is the golden
  ratio $ (\sqrt{5}+1)/2 $. It is well-known that
  \begin{displaymath}
    \phi = [1,1,\ldots] = 1 + \frac{1}{1 + \frac{1}{1 + \frac{1}{1 + \ddots}}}
  \end{displaymath}
  and so by \cref{prp:farey_expansion} we get that $ 1/\phi $ is approximated by
  \begin{displaymath}
    [1] = 1,\quad [1,1] = 1/2,\quad [1,1,1] = 2/3,\ldots;
  \end{displaymath}
  in other words, by the Fibonacci fractions $ \fib(q-1)/\fib(q) $. We have already computed the
  corresponding polynomials, which are listed in \cref{tab:fareysfib};
  the inverse images of $ -2 $ under the first 16 such polynomials are shown in \cref{fig:fibloci},
  where the $ 1/\phi$-cusp is approximated by the `corner' point, at the top-left of each picture.
  Quite a good approximation seems to be given after only a few pictures. Observe that the shape is somehow stable:
  \textit{a priori} it is just the cusps which must converge to a point, but we see that the entire set of roots seems
  to be converging to the solid outline of some fractal shape.
\end{ex}

\begin{figure}
  \centering
  \begin{subfigure}{.24\textwidth}
    \centering
    \includegraphics[width=\textwidth]{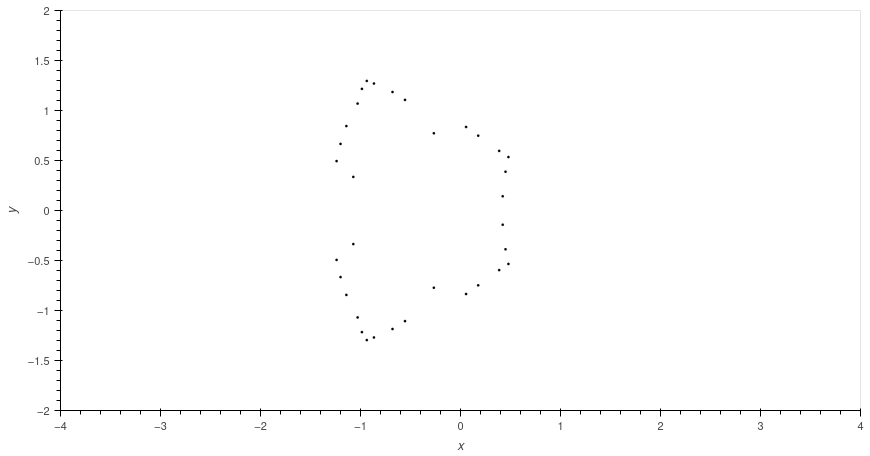}
    \caption{$21/34$}
  \end{subfigure}
  \begin{subfigure}{.24\textwidth}
    \centering
    \includegraphics[width=\textwidth]{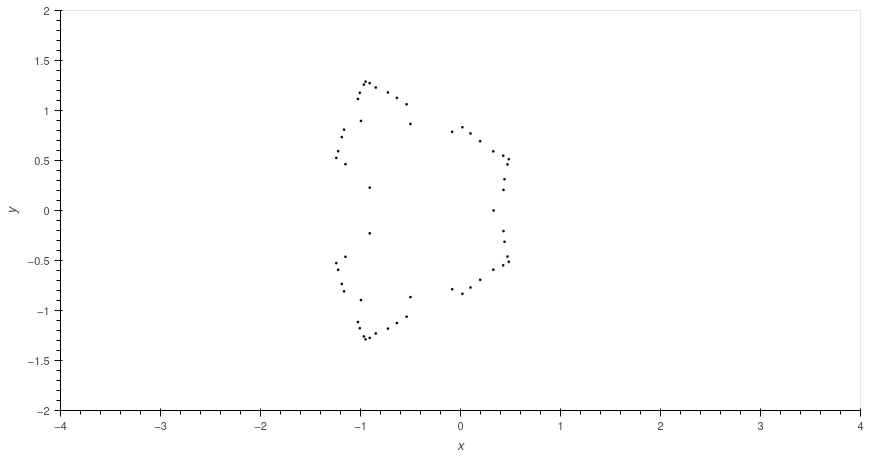}
    \caption{$34/55$}
  \end{subfigure}
  \begin{subfigure}{.24\textwidth}
    \centering
    \includegraphics[width=\textwidth]{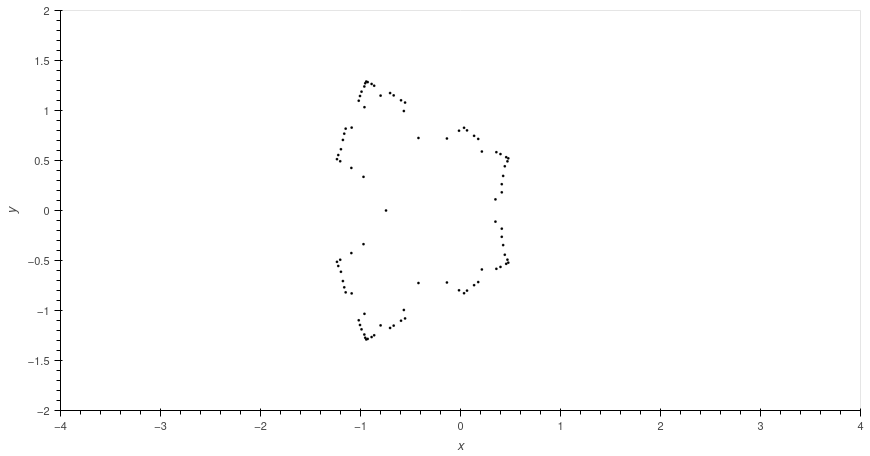}
    \caption{$55/89$}
  \end{subfigure}
  \begin{subfigure}{.24\textwidth}
    \centering
    \includegraphics[width=\textwidth]{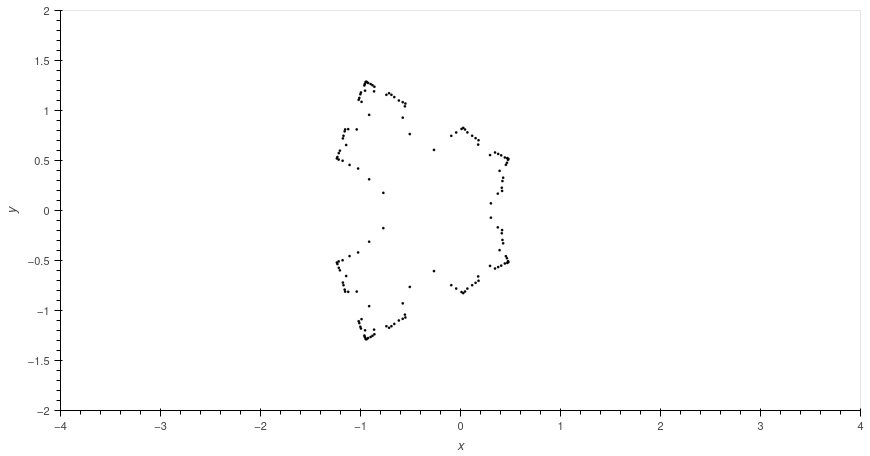}
    \caption{$89/144$}
  \end{subfigure}\\
  \begin{subfigure}{.24\textwidth}
    \centering
    \includegraphics[width=\textwidth]{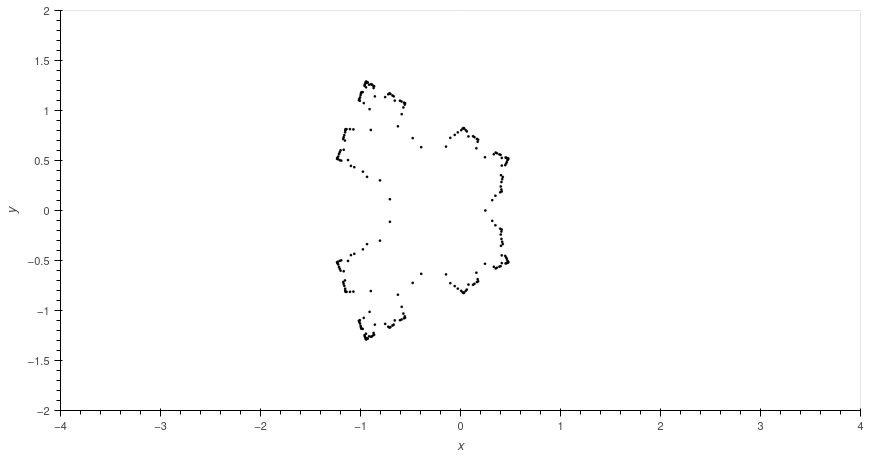}
    \caption{$144/233$}
  \end{subfigure}
  \begin{subfigure}{.24\textwidth}
    \centering
    \includegraphics[width=\textwidth]{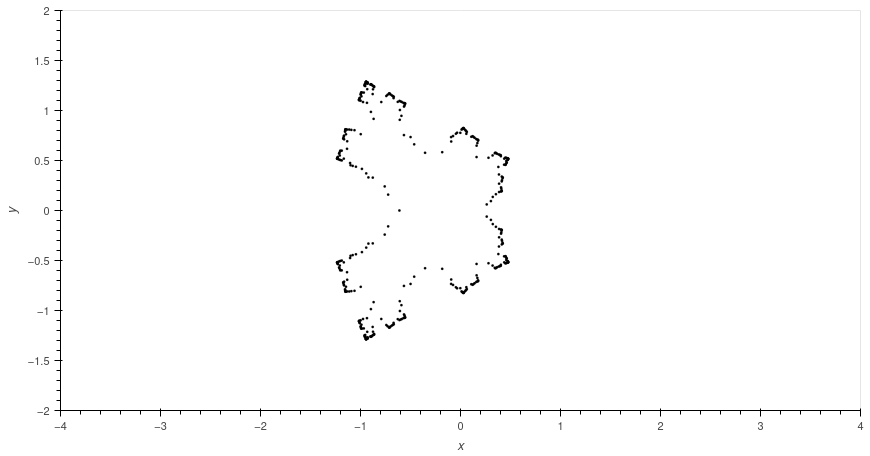}
    \caption{$233/377$}
  \end{subfigure}
  \begin{subfigure}{.24\textwidth}
    \centering
    \includegraphics[width=\textwidth]{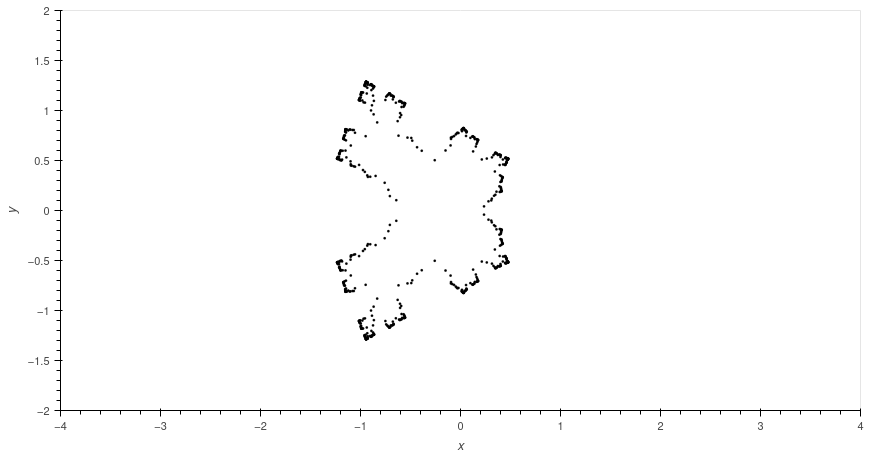}
    \caption{$377/610$}
  \end{subfigure}
  \begin{subfigure}{.24\textwidth}
    \centering
    \includegraphics[width=\textwidth]{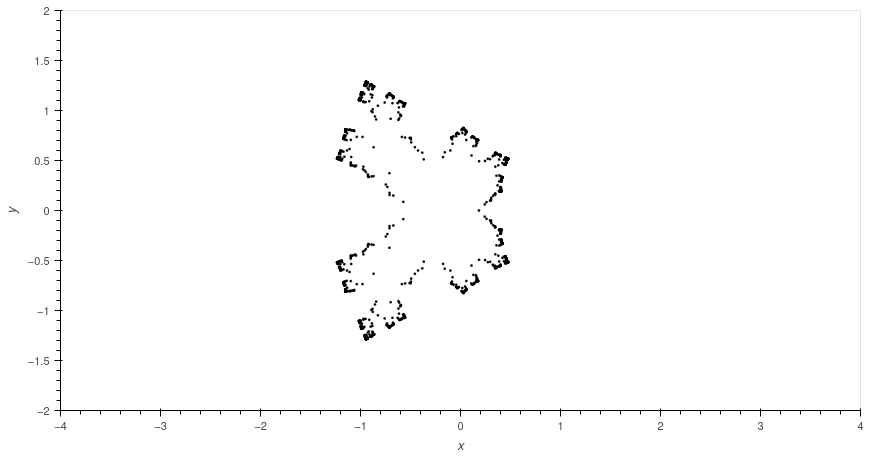}
    \caption{$610/987$}
  \end{subfigure}
  \caption{The zeros of $ \Phi_{\fib(p-1)/\fib(p)} $ for $ p \in \{8,\ldots,16\} $.\label{fig:fibloci}}
\end{figure}

\begin{ex}[$\sqrt{2}$]
  We consider another irrational number with a nice continued fraction decomposition:
  \begin{displaymath}
    \sqrt{2} = [1,2,2,\ldots,2,\ldots]
  \end{displaymath}
  so we may approximate $ 1/\sqrt{2} $ by the Farey sequence beginning
  \begin{displaymath}
    1, 2/3, 5/7, 7/10, 12/17, \ldots .
  \end{displaymath}
  The corresponding Farey polynomials are higher degree and so longer to write down than the Fibonacci polynomials, but it is easy enough
  for the computer to plot the preimages of $-2$ for the first 10 or so polynomials; we give the preimage for $ \Phi_{408/577} $ in \cref{fig:sqrtpic}.
\end{ex}

\begin{figure}
  \centering
  \begin{subfigure}{0.45\textwidth}
    \centering
    \includegraphics[width=\textwidth]{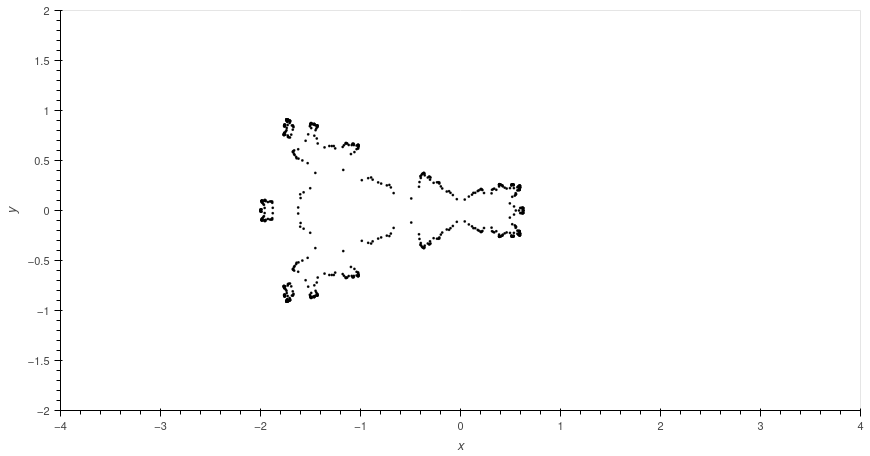}
    \caption{$ \left(\Phi^{\infty,\infty}_{408/577}\right)^{-1}(-2) $, approximating the $ 1/\sqrt{2} $ cusp point.\label{fig:sqrtpic}}
  \end{subfigure}%
  \begin{subfigure}{0.45\textwidth}
    \centering
    \includegraphics[width=\textwidth]{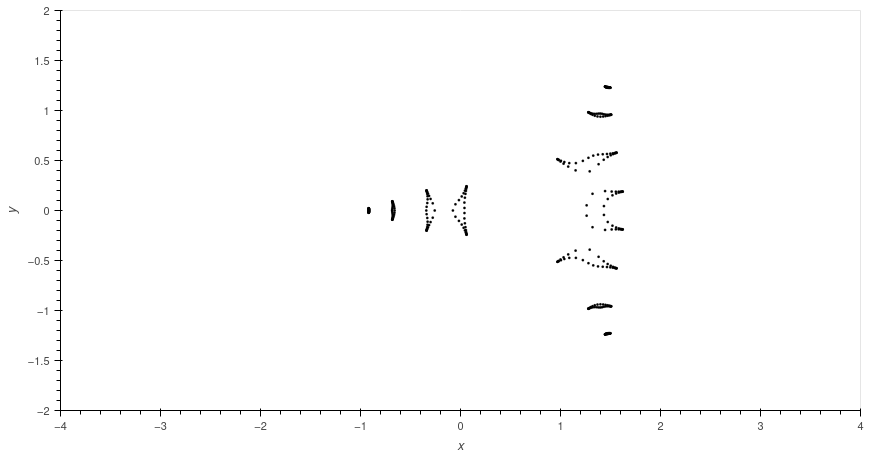}
    \caption{$ \left(\Phi^{\infty,\infty}_{113/355}\right)^{-1}(-2) $, approximating the $ 1/\pi $ cusp point.\label{fig:pipic}}
  \end{subfigure}
  \caption{Two more root sets arising from approximations of irrational cusps.}
\end{figure}

\begin{ex}[$\pi$]
  We give one final example since the root structure seems qualitatively quite different to the previous two. Considering the expansion
  \begin{displaymath}
    \pi = [3,7,15,1,292,1,1,1,2,\ldots]
  \end{displaymath}
  we may approximate $ 1/\pi $ by the Farey sequence beginning
  \begin{displaymath}
    1/3, 7/22, 106/333, 113/355, 33102/103993, 33215/104348, \ldots .
  \end{displaymath}
  We give the preimage for $ \Phi_{113/355} $ in \cref{fig:pipic}. We are unable to explain why the roots of the polynomial here do not seem to `fill out'
  the boundary of a fractal shape, like in the previous examples, but instead coalesce into the boundaries of loops.
\end{ex}

We end by describing a related family of polynomials. The recursion \cref{eq:recrelation} induces a dynamical system on a field $ k $, namely $ f : k^3 \to k^3 $ by
\begin{displaymath}
  f(x^1,x^2,x^3) = (x^2,x^3, 8 - x^1 - x^2 x^3).
\end{displaymath}

By replacing each $ x_i $ with $ y_i \coloneq x_i-2 $ (since $2$ is a fixed point), we may replace \cref{eq:recrelation} with the homogeneous relation
\begin{equation}\label{eq:homrel}
  y_i = -y_{i-3} - y_{i-2} y_{i-1} - 2(y_{i-2}+y_{i-1});
\end{equation}
with the initial conditions $ (y_1,y_2,y_3) = (-z,z,z^2) $, the result is a sequence of \textbf{reduced Farey polynomials} $ \phi_{p/q} := \Phi^{\infty,\infty}_{p/q} - 2 $
which we studied in \autocite{ems21} (in that paper, these polynomials were called $ Q_{p/q} $). Even a few minutes spent looking at these polynomials computationally shows
that they have a much nicer factorisation structure than the Farey polynomials, and characterising this factorisation structure precisely seems to be an interesting project.

For the case of two parabolic generators, one can predict a precise form for the reduced Farey polynomial in terms of the Riley polynomial (see remark on page \pageref{rem:riley}), which has an easier recursion rule: Chesebro \autocite[Section 5]{chesebro19} has introduced extensive vocabulary to deal with recursions of this type. Indeed, as explained in \autocite[Section 2.1]{akiyoshi}, the once-punctured torus $ T_* $ and the four-times punctured sphere are finitely covered by the $(0;2,2,2,\infty)$-orbifold $ O $: the diagram $ T_* \to O \leftarrow S_{0,4} $ is known as the Fricke diagram \autocite{sheingorn85}. By lifting a representation $\pi_1(T_*)\to \mathrm{PSL}(2,\C)$ to $\pi_1(O)$ and then descending it to $\pi_1(S_{0,4})$ one obtains the relationship between Farey and Riley polynomials (up to a factor, the reduced Farey polynomial $ \phi_{p/q} $ is a square root of the Riley polynomial $ \Lambda_{p/q} $). It would be an interesting extension of this work to give a similar proof for the case of elliptic generators: we have significant computational evidence suggesting that elliptic Farey polynomials also have a `square root'. There has been extensive work on the algebraic and number-theoretic interpretations of these polynomials in the punctured torus case---see for instance Riley \autocite{riley72,riley75}, Series \autocite{series85}, and Bowditch \autocite{bowditch98}---and we believe that trying to find versions of this theory which allow for torsion groups may have similarly pretty connections with other branches of mathematics.

\sloppy
\printbibliography

\end{document}